\documentclass[12pt,a4paper]{article}
\usepackage{hyperref}
\usepackage{mathtools}
\usepackage[nottoc,numbib]{tocbibind}
\usepackage[english]{babel}
\usepackage[utf8x]{inputenc}
\usepackage{amsfonts,longtable,amssymb,amsmath,latexsym,dsfont,mathabx}
\usepackage{wrapfig}
\usepackage[title]{appendix}
\usepackage{amsthm, amsmath, amssymb, amsfonts, amscd, amscd, geometry}
\newsavebox{\ssa}
\usepackage{bbm}
\usepackage{dsfont}
\usepackage{tikz-cd}

\newtheorem{thm}{Theorem}[subsection]
\newtheorem{prop}[thm]{Proposition}
\newtheorem{lemma}[thm]{Lemma}
\newtheorem{cor}[thm]{Corollary}

\newtheorem{conj}[thm]{Conjecture}

\theoremstyle{definition}
\newtheorem{def0}[thm]{Definition}

\newtheoremstyle{boldremark}
    {\dimexpr\topsep/2\relax} 
    {\dimexpr\topsep/2\relax} 
    {}          
    {}          
    {\bfseries} 
    {.}         
    {.5em}      
    {}          

\theoremstyle{boldremark}
\newtheorem{rem}[thm]{Remark}
\newtheorem{ex}[thm]{Example}
\newtheorem{cons}[thm]{Construction}

\newcommand{\lthm}{Łoś's theorem}
\newcommand{\Rep}{\textrm{Rep}}
\newcommand{\Repb}{\textbf{Rep}}
\newcommand{\Hom}{\textrm{Hom}}
\newcommand{\End}{\textrm{End}}
\newcommand{\Id}{\textrm{Id}}

\newcommand{\Ind}{\textrm{Ind}}
\newcommand{\Tr}{\textrm{Tr}}

\newcommand{\Img}{\textrm{Im}}

\newcommand{\IND}{\textrm{IND}}

\newcommand{\fan}{for almost all $n$}
\newcommand\prodF{{\prod}_{\mathcal F}}
\newcommand{\Fp}{\overline{\mathbb F}_p}
\newcommand{\Fpn}{\overline{\mathbb F}_{p_n}}
\newcommand{\FQ}{\overline{\mathbb Q}}
\newcommand{\Cext}{\overline{\mathbb C(\nu)}}
\newcommand{\chct}{\textrm{char}}

\newcommand{\ct}{\textrm{ct}}

\newcommand{\enalg}{\End(\Bbbk^r)}

\long\def\/*#1*/{}

\title{Deformed Double Current Algebras via Deligne Categories}

\author{Daniil Kalinov}
\date{}

\begin{document}

\maketitle

\begin{abstract}
 In this paper we give an alternative construction of a certain class of Deformed Double Current Algebras. These algebras are deformations of $U(\enalg[x,y])$ and they were initially defined and studied by N.Guay in his papers. Here we construct them as algebras of endomorphisms in Deligne category. We do this by taking an ultraproduct of spherical subalgebras of the extended Cherednik algebras of finite rank. 
\end{abstract}

\tableofcontents

\section{Introduction}

The Deformed Double Current Algebras were introduced by Nicolas Guay in the papers \cite{Gu1,Gu2,Gu3,guay2016deformed}. In these papers Guay gives several presentations of these algebras in terms of generators and relations, starting with the type A in \cite{Gu2} and then moving to any Lie algebra of rank $\ge 3$ in \cite{guay2016deformed}. This paper is concerned with an alternative construction of these algebras in type $A$, which provides us with the DDCA in the cases of $\mathfrak{gl}_1$, $\mathfrak{gl}_2$ and $\mathfrak{gl}_3$. This construction also provides us with an additional source of representations for these algebras and also, generally, gives us a new useful perspective on them. Note that another place there these algebras were studied is the paper \cite{costello2017holography} by Kevin Costello, where he constructs these algebras through the study of the algebra of ADHM quantum mechanics.

This paper is continuation of the research started in \cite{etingof2020new}. There the author of this paper together with Pavel Etingof and Eric Rains presented a way to construct the DDCA of rank 1 in type A and B as the ultraproduct of the spherical subalgebras of the corresponding Cherednik algebras, which can also be though of a spherical subalgebra of a Cherednik algebra defined in the Deligne category $\Rep(S_\nu)$. 

In the current paper this argument to the higher rank. In order to do this we use the notion of the Cherednik algebra extended to the higher rank introduced in \cite{etingof2020representations}. We transfer this notion from the finite rank setting to the setting of Deligne categories. There we define a spherical subalgebra of the Cherednik algebra $\widetilde{\mathcal D}_{t,k\nu}(r)$. This construction automatically induces a structure of a representation of $\widetilde{\mathcal D}_{t,k\nu}(r)$ on any homomorphism space $\Hom_{\Rep(S_\nu)}(\mathbb C, M)$, where $M$ is a representation of the extended Cherednik algebra in the Deligne category. The current paper thus continues the trend of working with Deligne categories using ultrafilters that was initiated in \cite{deligne2007categorie,harman2016deligne} and was carried out, for example, in \cite{kalinov2018finite,harman2020classification,etingof2020new, utiralova2020harish}.

Then we introduce a slightly different algebra $\mathcal D_{t,k}(r)$ by making the parameter $\nu$ of the previous algebra into a central element. This algebra turns out to be isomorphic to Guay's DDCA for $r \ge 4$ and $t +rk \ne 0$, which we prove in the final section of this paper.

The structure of the paper is as follows. In Section 2 we discuss the preliminary notation and give a quick overview of the notion of  ultrafilters and ultraproducts used extensively in this paper. In Section 3 we introduce the Deligne Categories and show how they can be approached through the lens of ultraproducts, and study the structure of the symmetric power of a unital algebra $S^\nu(A)$ as an object of the Deligne category $\Rep(S_\nu)$. In section 4 we give a definition of the extended Cherednik algebra in  finite rank, construct a generating set of its spherical subalgebra and then extend the definition to  complex rank. In Section 5 we construct the DDCA  $\mathcal D_{t,k}(r)$ and its basis. In Section 6 we prove that this DDCA is isomorphic to the DDCA constructed by Guay.

{\bf Acknowledgments.} 
I would like to thank Pavel Etingof and Nicolas Guay for conversations we had about the content of this paper.
The work on this paper was partially supported by the NSF grant DMS - 1916120.

\section{Preliminaries and notation}

\subsection{General notation} \label{sectnot}

In what follows we will use a lot of different categories of representations. We will always denote the usual (``finite rank") categories of representations using the boldface font, and use the regular font for the interpolation categories (e.g. $\Rep(S_\nu)$).

For example we will use the following notation for the categories of representations of symmetric groups. For convenience set ${\mathbb F}_0 = \Bbb Q$.

\begin{def0}
By $\textbf{Rep}(S_n; \Bbbk)$ denote the category of (possibly infinite dimensional) representations of the symmetric group $S_n$ over $\Bbbk$. By $\textbf{Rep}^{f}(S_n;\Bbbk)$ denote the full subcategory of finite dimensional representations. 
Also for $p\ge 0$ set  $\textbf{Rep}_p(S_n) := \textbf{Rep}(S_n; \Fp)$ and  $\textbf{Rep}^{f}_p(S_n) := \Repb^f(S_n; \Fp)$. 
\end{def0}

We will also fix the notation for working with Young diagrams and for the irreducible representations of the symmetric group.
\begin{def0}
 For a Young diagram $\lambda$, by $l(\lambda)$ denote the number of rows of the diagram (the length), by $|\lambda|$ the number of boxes (the weight) and by $\ct(\lambda)$ the content of $\lambda$, i.e., $\ct(\lambda)=\sum_{(i,j)\in \lambda}(j-i)$, where $(i,j)$ denotes the box of $\lambda$ in row $i$ and column $j$.
\end{def0}
\begin{def0}
For $p=0$ or $p>n$ and a Young diagram $\lambda$ such that $|\lambda|=n$ denote by $X_p(\lambda)$ the unique simple object of $\Repb_p(S_n)$ corresponding to $\lambda$.

For $n>0$ and $p\ge 0$ denote by $\mathfrak{h}_n^p \in \Repb_p(S_n)$, or shortly by $\mathfrak{h}_n$ (if there is no ambiguity about the characteristic) the standard permutation representation of $S_n$. 
\end{def0}

There is an important central element in $\Bbbk[S_n]$:
\begin{def0}\label{centraldef}
Denote the central element $\sum_{1 \le i<j \le n}s_{ij} \in \Bbbk[S_n]$ by $\Omega_n$. 
\end{def0}
\begin{rem}
Note that $\Omega_n$ acts on $X_p(\lambda)$ by $\ct(\lambda)$.
\end{rem}

As another piece of notation, below we will frequently use the following operation on Young diagrams:
\begin{def0} \label{defyoung}
 For a Young diagram $\lambda$ and an integer $n \ge \lambda_1+|\lambda|$ denote by $\lambda|_n$ the Young diagram $(n - |\lambda|, \lambda_1, \dots, \lambda_{l(\lambda)})$, where $\lambda_i$ is the length of the $i$-th row of $\lambda$.
\end{def0}

In what follows we will often use the language of tensor categories. Here's what we mean by a tensor category (see Definition 4.1.1 in \cite{etingof2016tensor}):
\begin{def0}
A tensor category $\mathcal C$ is a $\Bbbk$-linear locally finite abelian rigid symmetric monoidal category, such that $\End_{\mathcal C}(\mathds{1}) \simeq \Bbbk$.
\end{def0}

We will also fix a notation for the symmetric structure:
\begin{def0}
 For two objects $X,Y$ of a tensor category $\mathcal C$, we will denote by $\sigma_{X,Y}$ the map from $X \otimes Y$ to $Y \otimes X$, given by the symmetric structure. Oftentimes, when the objects we are referring to are obvious from the context, we will denote it simply by $\sigma$. Especially in the case when $X=Y$.
\end{def0}

We will also use the notion of the ind-completion of a category. For a general category ind-objects are given by diagrams in the category, with morphisms being morphisms between diagrams. However, in the case of a semisimple category there is a more concrete description. 

\begin{def0} \label{inddef}
For a semisimple category $\mathcal C$ with the set of simple objects $\{V_{\alpha}\}$ for $\alpha \in A$ the category\footnote{We use all uppercase letters to denote $\IND$, so as not to confuse it with the induction functors.}
 $\IND(\mathcal C)$  is the category $\mathcal D$ with objects $\bigoplus_{\alpha \in A} M_{\alpha}\otimes V_{\alpha}$, where $M_{\alpha}$ are (possibly infinite dimensional) vector spaces. The morphism spaces are given by:
$$
\Hom_{\mathcal D}(\bigoplus_{\alpha \in A} M_{\alpha} \otimes V_{\alpha},\bigoplus_{\beta \in A} N_{\beta} \otimes V_{\beta}) = \prod_{\alpha \in A} \Hom_{\textrm{Vect}}(M_{\alpha},N_{\alpha})  . 
$$
\end{def0}

Thus, in this case, we can think of ind-objects as infinite direct sums of objects of $\mathcal C$.

Next we would like to explain a way to define an ind-object of $\mathcal C$.
\begin{cons} \label{indcons}
Suppose $0 = X_0 \subset X_1 \subset X_2 \subset \dots \subset X_i \subset \dots$ is a nested sequence of objects of $\mathcal C$. Then their formal colimit, which we denote by $X$, is an object of $\IND(\mathcal C)$. We can write it down explicitly in terms of Definition \ref{inddef}. 

Indeed, suppose we have $X_i = \bigoplus_{\alpha \in A}M_{i,\alpha} \otimes V_{\alpha}$. Then it follows that:
$$
\bigcup_{i \in \mathbb N} X_i  = X = \bigoplus_{\alpha \in A} \left(\bigcup_{i \in \mathbb N} M_{i,\alpha}\right) \otimes V_{\alpha}  ,
$$
where $\bigcup_{i \in \mathbb N}X_i=\varinjlim X_i$ stands for the colimit along the diagram consisting of points numbered by $\mathbb N$ and arrows from $i$ to $i+1$ for all $i$.
\end{cons}

\begin{rem}\label{remindmor}
Suppose that $X$ and $Y$ are two objects constructed via Construction \ref{indcons}. Then:
$$
\Hom_{\IND(\mathcal C)}(X,Y) =\varprojlim_{i \in \mathbb N} \bigcup_{j \in \mathbb N}\Hom_{\mathcal C}(X_i,Y_j) .
$$

In case when $X$ is actually an object of $\mathcal C$, this simplifies to:
$$
\Hom_{\IND(\mathcal C)}(X,Y) = \bigcup_{j \in \mathbb N}\Hom_{\mathcal C}(X,Y_j).
$$
In other words, $X$ is a compact object of $\IND(\mathcal C)$.
\end{rem}
\begin{ex}
We have  $\textbf{Rep}_p(S_n) = \IND(\textbf{Rep}^{f}_p(S_n))$. Indeed, this holds for the representation category of any finite dimensional algebra.
\end{ex}

We will also use a notion of a bifiltered algebra below. 
\begin{def0}
A bifiltered vector space $V$ is a vector space together with a collection of subspaces $F^{i,j}V$ for $i,j \in \mathbb Z_{\ge 0}$ such that $F^{i,j}V \subset F^{i+1,j}V$ and $F^{i,j}V \subset F^{i,j+1}V$, and there exists a basis of $V$ such that the intersection of this basis with $F^{i,j}V$ gives a basis of $F^{i,j}V$ (i.e. the filtrations $F^{i, \cdot}V$ and $F^{\cdot,j}V$ are compatible).

A bifiltered algebra $A$ is an algebra which is bifiltered as a vector space such that $F^{i,j}A \cdot F^{i',j'}A \subset F^{i+i',j+j'}A$.
\end{def0}

This structure also induces a few standard filtrations:

\begin{rem}
Notice that bifiltered structure on $A$ induces two filtrations on $A$ via restriction. The first one is given by $F_h^iA = F^{i,\bullet}$ and we will call it the horizontal filtration of $A$. The second one is given by $F_v^iA = F^{\bullet,i}A$ and we will call it the vertical filtration of $A$.

There is another filtration on $A$ that we will call the total filtration on $A$. It is given by $F_t^lA = \bigcup_{i+j=l}F^{i,j}A$.
\end{rem}

Also it's easy to see that to specify a bifiltration it is enough to specify the horizontal and vertical degree of each generator of $A$.
\subsection{Ultrafilters}

Below we will briefly discuss some basic facts about ultrafilters and ultraproducts. Ultrafilters provide us with a notion of the limit of algebraic structures, which works really well for describing Deligne categories. Thus, we will use this framework extensively in the present paper.

We will only give brief definitions and look at a few examples here. For a much more elaborate discussion see Section 2.5 in \cite{etingof2020new}. The reader unfamiliar with this technology is advised to read corresponding section in that paper first. For more details on  this topic in the algebraic context, see \cite{schoutens2010use}.

\subsubsection{Ultrafilters and ultraproducts}

First we will fix an ultrafilter we will be considering below.

\begin{def0}
For the rest of the paper we will denote by $\mathcal F$ a fixed non-principal ultrafilter on $\mathbb N$.
\end{def0}

Throughout the paper we will use the following shorthand phrase.
\begin{def0}
By the statement  ``$A$ holds for almost all $n$'', where $A$  is a logical statement depending on $n$, we will mean that $A$ is true for some subset of natural numbers $U$, such that $U \in \mathcal F$.
\end{def0}

Now, define the notion of an ultraproduct. 
\begin{def0}
Suppose we have a sequence of sets $E_n$ labeled by natural numbers. Consider the set $\prod'_{\mathcal F}E_n$ consisting of the sequences $\{e_n\}_{n \in A}$ for a set $A \in \mathcal F$ and $e_n \in E_n$. i.e., $\prod'_{\mathcal F}E_n$ consists of sequences of elements of $E_n$ which are defined \fan. Then $\prod_{\mathcal F}E_n$ is the quotient of $\prod'_{\mathcal F}E_n$ by the following relation: $\{e_n\}_{n \in A} \sim \{e_n'\}_{n \in A'}$ iff $e_n = e_n'$ for almost all $n$ (i.e.,  on $B \subset A' \cap A$, such that $B \in \mathcal F$).
The set $\prod_{\mathcal F}E_n$ is called the ultraproduct of the sequence $\{E_n\}_{n \in \mathbb N}$. 
\end{def0}

Oftentimes we use the following notation:
\begin{def0}
For a sequence $\{E_n\}_{n \in \mathbb N}$, denote an element $\{ e_n \}_{n \in \mathbb N} \in \prod_{\mathcal F}E_n$ by $\prod_\mathcal F e_n$.
\end{def0}

This construction is interesting for us, because it, in a certain sense, preserves a lot of algebraic structures. This can be formalized into the following theorem: 
\begin{thm}\textbf{\lthm} (Theorem 2.3.2 in \cite{schoutens2010use})

Suppose we have a collection of sequences of sets $E^{(k)}_i$ for $k = 1,\dots,m$, a collection of sequences of elements $f^{(r)}_i$ for $r = 1,\dots, l$, and a  formula of a first order language $\phi(x_1,\dots,x_l, Y_1, \dots, Y_m)$ depending on some parameters $x_i$ and sets $Y_j$. Denote by \linebreak $E^{(k)} = \prod_{\mathcal F}E^{(k)}_{n}$ and $f^{(r)} = \prod_{\mathcal F} f^{(r)}_n$.  Then 
$\phi(f^{(1)}_n, \dots, f^{(l)}_n, E^{(1)}_n, \dots, E^{(m)}_n)$ is true for almost all $n$ iff  $\phi(f^{(1)}, \dots, f^{(l)}, E^{(1)}, \dots E^{(m)})$ is true.
\end{thm}

Below we consider a few examples of ultraproducts that are important for our future constructions.

\begin{ex} \label{fieldextrans}
 Take the ultraproduct of a countably infinite number of copies of $\overline{\mathbb Q}$. By \lthm $\prod_{\mathcal F} \overline{\mathbb Q}$ is a field, which is algebraically closed. It has characteristic zero since $\forall k  \in \mathbb Z$ such that $k\ne 0$ it follows that $ k = \prod_{\mathcal F} k\ne 0$.  Also it is easy to see that its cardinality is continuum. Hence by Steinitz's theorem\footnote{This theorem tells us that two uncountable algebraically closed fields are isomorphic iff their characteristic and cardinality are the same. It is proven in \cite{steinitz1910algebraische}.} $\prod_{\mathcal F} \overline{\mathbb Q} \simeq \mathbb C$. Note that there is no canonical isomorphism.
 
 Consider the ultraproduct of integers $\prodF n$. Via the isomorphism constructed in the previous paragraph this is an element of $\mathbb C$. Notice that this element cannot satisfy any nontrivial polynomial equation over $\mathbb Q$ (indeed, the corresponding polynomial must have infinitely many roots), hence $\prodF n$ is a transcendental element of $\mathbb C$. By an automorphism of $\mathbb C$ we can send this element into any transcendental element of $\mathbb C$. 
 
 Thus we conclude that for any transcendental element $\nu \in \mathbb C$ there is an isomorphism $\prodF \FQ \simeq \mathbb C$, such that $\prodF n = \nu$.
 
 Also notice that by Steinitz's theorem it follows that $\overline{\mathbb C(x)}\simeq \mathbb C$, since they have the same cardinality. Thus we can also conclude that there is an isomorphism $\prodF \FQ \simeq \overline{\mathbb C(x)}$ such that $\prodF n = x$.
\end{ex}
 
 \begin{ex} \label{fieldexalg}
 Take the ultraproduct of $\overline{\mathbb F}_{p_n}$ for some sequence of distinct prime numbers $p_n$. As before, by \lthm${}$ $\prod_{\mathcal F} \overline{\mathbb F}_{p_n}$ is a field, which is algebraically closed. Also it has cardinality continuum. Now for any natural number $k$, we have $k = \prod_{\mathcal F} k \ne 0$, since it is equal to zero for at most a finite number of $n$.  Hence $\prod_{\mathcal F} \overline{\mathbb F}_{p_n} \simeq \mathbb C$ by Steinitz's theorem, again not in a canonical way.
 
 Suppose we are given an algebraic number $\nu \in \mathbb C$. We can also pick $\nu_n$ and $p_n$ such that $\nu_n < p_n$ and  $\prodF \nu_n = \nu$ inside $\prodF \Fpn \simeq \mathbb C$. For details see Example 2.5.14 in  \cite{etingof2020new}.
\end{ex}

\begin{ex}\label{catex}
 Suppose $\mathcal C_n$ is a sequence of  (locally small) categories. We can define the ultraproduct category $\widehat{\mathcal C} = \prod_{\mathcal F} \mathcal C_n$ as the category 
whose objects are sequences of objects in $\mathcal C_n$.  
For clarity we will denote the ultraproduct of objects by\footnote{The superscript $C$ stands for "category".} $\prodF^C$. The morphisms in $\widehat{\mathcal C}$ are given by
 $$
 \Hom_{\widehat{\mathcal C}}(\prodF^C X_n,\prodF^C Y_n) = \prodF \Hom_{\mathcal C_n}(X_n,Y_n), 
 $$
 and the composition maps are given by the ultraproducts of the composition maps, i.e., \linebreak $(\prod_{\mathcal F}f_n) \circ (\prod_{\mathcal F}g_n) = \prod_{\mathcal F} (f_n \circ g_n)$. By \lthm${}$ this data satisfies the axioms of a category. If the categories $\mathcal C_n$ have some structures, for example the structures of an abelian or monoidal category, then $\widehat{\mathcal C}$ also has these structures\footnote{But the finite-length property, for example, does not survive, as it cannot be formulated as a first-order logical statement.}.
 
 Usually $\widehat{\mathcal C}$ is too big and it is interesting to consider a certain full subcategory $\mathcal C$ in there, for example by only  considering the ultraproducts of  sequences of objects of $\mathcal C_i$ bounded in some sense. This will be discussed in more detail in the next subsection. 
\end{ex}

\subsubsection{Restricted ultraproducts}

When one works with a sequence of objects which are in some sense infinite dimensional, it's sometimes useful to consider a subobject in the ultraproduct consisting of the sequences of elements which are bounded in a certain way. This can be called a {\it restricted ultraproduct.} We have already mentioned this in the case of categories in Example \ref{catex}. For example, the Deligne category $\Rep(S_\nu)$ will be constructed as a full subcategory in a certain ultraproduct category. 

In this section we will outline the definitions of the restricted ultraproduct which makes sense in the case of filtered or graded vector spaces and categories. For more information and examples of this construction see section 2.5.3 in \cite{etingof2020new}.

\begin{def0}\label{restrdef}
For a sequence of vector spaces $E_n$  with an increasing filtration \linebreak $F^0E_n \subset F^1E_n \subset \dots \subset F^kE_n \subset \dots$, define the restricted ultraproduct $\prodF^r E_n$ to be equal to $\bigcup_{k=0}^\infty \prodF F^kE_n  \subset \prodF E_n$.
\end{def0}

We will use this notion in the case when the dimensions of the space $F^kE_n$ are finite and stabilize as $n\to \infty$ for fixed $k$. Let us give a few examples.

\begin{ex}
Consider a countable-dimensional vector space $V$ over $\Bbbk$. Consider a sequence of copies of $V$, i.e., $V_n = V$. Also consider an increasing filtration $F^jV$ by finite dimensional subspaces  and the same filtration on all $V_n$. We can calculate the restricted ultraproduct of this sequence:
$$
\prodF^rV_n = \bigcup_{k=0}^\infty \prodF F^kV_n = \bigcup_{k=0}^\infty F^kV = V  .
$$
Whereas the usual ultraproduct $\prodF V_n$ is more than countable-dimensional.
\end{ex}

We also would like to introduce a related construction, which we will also call a restricted ultraproduct. This will take place in the setting of ultraproducts of categories. Suppose $\{\mathcal D_i\}$ is a sequence of artinian abelian categories and $\mathcal D = \prodF \mathcal D_i$ is their ultraproduct (an abelian category which is, in general, not artinian). Suppose $\mathcal C$ is a full artinian subcategory of $\mathcal D$. Using Construction \ref{indcons} we can obtain ind-objects of $\mathcal C$ in the following way.
\begin{cons} \label{indconsult}
Suppose we have a sequence of ind-objects $X_n \in \IND(\mathcal D_n)$ such that each $X_n$ is equipped with a filtration by objects of $\mathcal D_n$. I.e., we have the following sequence of inclusions
$$
F^0X_n \subset F^1X_n \subset \dots \subset F^iX_n \subset \dots \ ,
$$
where all $F^iX_n \in \mathcal D_n$ and $X_n = \bigcup_{i \in \mathbb N}F^iX_n$. Also suppose that for each $i\ge 0$, we have $\prodF^CF^iX_n  \in \mathcal C$. Denote $\prodF^CF^iX_n$ by $F^iX_\infty$. It is clear that 
we have injections $F^iX_\infty\hookrightarrow F^{i+1}X_\infty$. 

It follows that the sequence $F^iX_\infty$ defines an object $X_\infty \in \IND(\mathcal C)$ as:
$$
X_\infty = \bigcup_{i \in \mathbb N}F^iX_\infty = \bigcup_{i \in \mathbb N}\prodF^CF^iX_n  .
$$
\end{cons}

We will use a special notation for this construction:
\begin{def0} \label{restrcatdef}
In the setting of Construction \ref{indconsult}, call $X_\infty$ the restricted ultraproduct of $X_n$ with respect to the fixed filtration. We will write
$$
X_\infty = \prodF^{C,r}X_n  .
$$
\end{def0}

\begin{rem} Note that if $\widetilde F^\bullet$ is another filtration on the sequence $\lbrace{X_n\rbrace}$ such that $\prodF^CF^iX_n  \in \mathcal C$, such that  for any $i$ there exist $r(i),s(i)$ such that 
$F^iX_n\subset \widetilde F^{r(i)}X_n$ and $\widetilde F^iX_n\subset F^{s(i)}X_n$ for almost all $n$. Then it follows that the restricted ultraproducts $\prodF^rX_n$ taken with respect to both filtrations are naturally isomorphic. See Remark 2.5.24 in \cite{etingof2020new} for more information on this.
\end{rem}

\begin{rem}
Note that we can easily define the restricted ultraproduct of a series of bifiltered algebras $A_n$, with finite-dimensional filtration components as $\bigcup_{i,j\ge 0} \prodF F^{i,j}A_n$. Note that the result is the same as the restricted product taken with respect to the total filtration of $A_n$.

The same goes for the sequence of bifiltered ind-objects of artinian categories similarly to Construction \ref{indconsult}.

Thus below we will use these two operations interchangeably. 
\end{rem}

\section{Deligne Categories}

\subsection{Constructions of the category $\Rep(S_\nu)$}

In this section we will very briefly discuss a well known construction of the interpolation category for the symmetric group due to Deligne \cite{deligne2007categorie} and its basic properties. For an extended version of this discussion see Section 3 in \cite{etingof2020new}. For more on this topic see \cite{comes2009blocks,comes2012ideals,comes2014deligne,etingof2014representation,etingof2016representation}. Anyone who encounters Deligne categories for the first time is advised to read one of the above papers first. We assume that $\Bbbk$ has characteristic $0$.

We can define the Deligne category $\Rep(S_\nu; \Bbbk)$ in the following manner:
\begin{def0}
 For $\nu \in \Bbbk$, the Deligne category $\Rep(S_\nu; \Bbbk)$ is the Karoubian envelope of the additive envelope of a certain skeletal monoidal category $\Rep^0(S_\nu; \Bbbk)$ defined using  certain combinatorial data. 
 \end{def0}
 
This definition won't be used much in the current paper, instead the reader can also think about Theorem \ref{ultdelthm} as the definition of Deligne categories.

Below we will list a few pieces of notation and results concerning Deligne categories. They are well known and can be found for example in \cite{comes2009blocks,etingof2014representation}.

\begin{def0}
The object $[1]$ is called the permutation representation and is denoted by $\mathfrak h$.
  The object $[0]$ is called the trivial representation and is denoted by $\Bbbk$ (by a slight abuse of notation).
\end{def0}

The important properties of $\Rep(S_\nu; \Bbbk)$ are listed below:
\begin{prop}
\textbf{a)} For $\nu \notin \mathbb Z_{\ge 0}$ ${\rm Rep}(S_\nu ; \Bbbk)$ is a semisimple tensor category. \\
\textbf{b)} For $\nu \notin \mathbb Z_{\ge 0}$  simple objects of ${\rm Rep}(S_\nu; \Bbbk)$ are in 1-1 correspondence with Young diagrams of arbitrary size. They are denoted by $\mathcal X(\lambda)$. Moreover $\mathcal X(\lambda)$ is a direct summand in $[|\lambda|]$. \\ 
\textbf{c)} The categorical dimension of $\mathfrak h$ is $\nu$ and of $\Bbbk$ is $1$. \\
\textbf{d)} All $\mathcal X(\lambda)$ are self-dual.
\end{prop}

The Deligne category enjoys a certain universal property:
\begin{prop}\label{unprop}
(8.3 in \cite{deligne2007categorie})
For any $\Bbbk$-linear Karoubian symmetric monoidal category $\mathcal T$, the category of $\Bbbk$-linear symmetric monoidal functors from ${\rm Rep}(S_\nu; \Bbbk)$ to $\mathcal T$ is equivalent to the category $\mathcal T^f_\nu$ of commutative Frobenius algebras in $\mathcal T$ of dimension $\nu$. The equivalence sends a functor $F$ to the object $F(\mathfrak h)$.
\end{prop}

The important consequence of this result is that for every commutative Frobenius algebra $A$ in a Karoubian symmetric category $\mathcal T$ of dimension $\nu$, we have a symmetric monoidal functor from $\Rep(S_\nu; \Bbbk)$ to $\mathcal T$ which sends $\mathfrak h$ to $A$.

\begin{rem} 
Here by a commutative Frobenius algebra in $\mathcal T$ we mean an object $A$ with the  following structure. It is an associative commutative algebra with the corresponding algebraic structure given by $\mu_A,1_A$, and if we define a map:
$$
\Tr: A \xrightarrow{1 \otimes {\rm coev}_A} A \otimes A \otimes A^* \xrightarrow{\mu_A \otimes 1} A \otimes A^* \xrightarrow{{\rm ev}_A} \mathds{1},
$$
then the pairing $A \otimes A \xrightarrow{\mu_A} A \xrightarrow{\Tr} \mathds{1}$ is required to be non-degenerate, i.e., it corresponds to an isomorphism between $A$ and $A^*$ under the identification of $\Hom_{\mathcal T}(A\otimes A,\mathds{1})$ with $\Hom_{\mathcal T}(A, A^*)$.
\end{rem}

In the rest of the paper we will use Deligne categories over the following fields:
\begin{def0} \label{extdef}
For $\nu \in \mathbb C$ set $\Rep(S_\nu) := \Rep(S_\nu;\mathbb C)$. And for $\nu \in \overline{\mathbb C(\nu)}$ set \linebreak $\Rep^{\rm ext}(S_\nu) := \Rep(S_\nu; \overline{\mathbb C(\nu)})$.
\end{def0}

\begin{rem}
Note that although $\mathbb C$ and $\overline{\mathbb C(\nu)}$ are isomorphic as fields, such isomorphism is not canonical. Thus it will be convenient to distinguish them in the following discussions.
\end{rem}

\subsection{Deligne category $\Rep(S_\nu)$ as an ultraproducts}

\subsubsection{The category $\Rep(S_\nu)$ as an ultraproduct}

In this section we will  show how to construct $\Rep(S_\nu)$ using ultraproducts, and discuss some important consequences of this construction. We will omit most of the proofs here. For a more elaborate version of this discussion see Section 3.2.1 in \cite{etingof2020new}. This method is very useful, because it allows one to transfer all kinds of constructions and their properties from the case of finite rank categories almost automatically. The main ideas of this approach were contained in \cite{deligne2007categorie},\cite{harman2016deligne}\footnote{For the similar discussion about $\Rep(GL_\nu)$ see \cite{deligne2007categorie}, \cite{harman2016deligne}, \cite{kalinov2018finite}.}.

The idea is to construct the category $\Rep(S_\nu)$ for non-integer $\nu$ as a full subcategory in the ultraproduct category following Example \ref{catex}. We have the following result (See the introduction of \cite{deligne2007categorie} or Theorem 1.1 in \cite{harman2016deligne}):

\begin{thm} \label{ultdelthm}

\textbf{a)} Suppose $\nu\in \mathbb C$ is transcendental. Consider $\widehat{\mathcal C} = \prod_{\mathcal{F}} \textbf{Rep}^f_0(S_n)$. Set $\mathfrak h_\nu:= \prod^C_{\mathcal{F}}\mathfrak h_n$. Fix an isomorphism $\prod_{\mathcal F}\overline{\mathbb Q}\simeq \mathbb C$ such that $\prod_{\mathcal F} i = \nu$. Then the full subcategory of the $\prod_{\mathcal F}\overline{\mathbb Q}$-linear category $\widehat{\mathcal C}$ generated by $\mathfrak h_\nu$ under taking tensor products, direct sums and direct summands is equivalent to the $\mathbb C$--linear category $\Rep(S_\nu)$, in a way consistent with the fixed isomorphism $\prod_{\mathcal F}\overline{\mathbb Q} \simeq \mathbb C$.

\textbf{b)} Suppose $\nu \in \mathbb C$ is algebraic but not a nonnegative integer. Fix a sequence of distinct primes $p_n$, a sequence of integers $\nu_n$, and an isomorphism $\prod_{\mathcal F}\overline{\mathbb F}_{p_n}\simeq \mathbb C$ such that \linebreak $\prod_{\mathcal F}\nu_n = \nu$. Set $\widehat{\mathcal C}: = \prod_{\mathcal{F}} \textbf{Rep}^f_{p_n}(S_{\nu_n})$. Set $\mathfrak h_\nu := \prod_{\mathcal{F}}^C\mathfrak h_{p_n}^{\nu_n}$. Then the  full subcategory of the $\prod_{\mathcal F}\overline{\mathbb F}_{p_n}$-linear category $\widehat{\mathcal C}$ generated by $\mathfrak h_\nu$ under taking  tensor products, direct sums and direct summands is equivalent to the $\mathbb C$-linear category $\Rep(S_\nu)$,  in a way consistent with the fixed isomorphism $\prod_{\mathcal F}\overline{\mathbb F}_{p_n}\simeq \mathbb C$.
\end{thm}

\begin{rem}
Note that for the purposes of this theorem we could also have used the categories $\Repb_{p_n}(S_{\nu_n})$.
\end{rem}

We can formulate a similar result for $\Rep^{\rm ext}(S_\nu)$:
\begin{cor} \label{ultdelcor}
Fix an isomorphism $\prodF \FQ \simeq \overline{\mathbb C(\nu)}$ such that $\prodF n = \nu$.  Set \linebreak $\widehat{\mathcal C} = \prod_{\mathcal{F}} \textbf{Rep}^f_{0}(S_{n})$. Set $\mathfrak h_\nu = \prod_{\mathcal{F}}^C\mathfrak h_{n}$. Then the  full subcategory of the $\prod_{\mathcal F}\FQ$-linear category $\widehat{\mathcal C}$ generated by $\mathfrak h_\nu$ under taking  tensor products, direct sums and direct summands is equivalent to the $\Cext$-linear category $\Rep(S_\nu)$,  in a way consistent with the fixed isomorphism $\prod_{\mathcal F}\FQ\simeq \Cext$.
\end{cor}

\begin{rem}
As mentioned in the beginning of Section \ref{sectnot}, to treat the algebraic and transcendental cases simultaneously, it's useful to agree on the convention that by $\overline{\mathbb F}_0$ we will mean $\overline{\mathbb Q}$, and so the case $\nu_n = n$, $p_n = 0$ in the setting of part $(b)$ of Theorem \ref{ultdelthm} gives us transcendental $\nu$. Also below we will always assume that the sequences $p_n$ and $\nu_n$ are the sequences from Theorem \ref{ultdelthm} or Corollary \ref{ultdelcor} corresponding to the given $\nu$. Finally, we will work only with $\nu \in \mathbb C \backslash \mathbb Z_{\ge 0}$.
\end{rem}

Now we would like to explain why this construction of the Deligne categories is quite useful. To begin with, we would like to construct the simple objects $\mathcal X(\lambda)$ as ultraproducts. This is easy to do, using the notation from Definition \ref{defyoung}:
\begin{prop} \label{simpprop}
The irreducible object $\mathcal X(\lambda)$ of ${\rm Rep}(S_\nu)$ can be obtained as an ultraproduct of irreducible objects of ${\bf Rep}^f_{p_n}(S_{\nu_n})$ as $\mathcal X(\lambda) = \prod_{\mathcal F}^CX_{\nu_n}(\lambda|_{\nu_n})$.
\end{prop}

This result allows us to reformulate the definition of $\Rep(S_\nu)$ as an ultraproduct in the following way.
\begin{prop} \label{simplerultrthm}
In the notation of Theorem \ref{ultdelthm} the category ${\rm Rep}(S_\nu)$ can be described as the full subcategory of $\widehat{\mathcal C} ={\bf Rep}^f_{p_n}(S_{\nu_n})$ consisting of sequences of objects $Y_n= \bigoplus_{\alpha \in A_n} X_{p_n}(\lambda_{n,\alpha})$ for some indexing sets $A_n$ and Young diagrams $\lambda_{n,\alpha}$ such that both the sequence of $|A_n|$ and the sequence of $\max_{\alpha \in A_n}(|\lambda_{n,\alpha}| - (\lambda_{n,\alpha})_1)$, where $(\lambda_{n,\alpha})_1$ is the length of the first row, are bounded {\fan}.
\end{prop}

We will also need to explain how to interpolate the central element $\Omega_n \in \Bbbk[S_n]$ to $\Rep(S_\nu)$. Recall that we can consider the central elements of $\Bbbk[S_{\nu_n}]$ as endomorphisms of the identity functor of $\Repb_{p_n}(S_{\nu_n})$.
\begin{def0} \label{delcentr} 
Denote by $\Omega$ the endomorphism of the identity functor of $\Rep(S_\nu)$ given by the restriction of the endomorphism $\prodF \Omega_{\nu_n}$.
\end{def0}
One can easily calculate the action of $\Omega$ on simple objects.
\begin{prop} 
\cite{etingof2014representation} The action of $\Omega$ on an object $\mathcal X(\lambda)$ is given by: $$
\Omega|_{\mathcal X(\lambda)} = \left(\ct(\lambda) - |\lambda| +\frac{(\nu - |\lambda|)(\nu - |\lambda|-1)}{2}\right) 1_{\mathcal X(\lambda)}  .
$$
\end{prop}

\begin{rem}
Note that all of the results of this Section work mutatis mutandis for $\Rep^{\rm ext}(S_\nu)$ (see Definition \ref{extdef}).
\end{rem}

Now we would like to give the reader a general idea of how this can be used to transfer constructions and facts from  representation theory in  finite rank to the context of Deligne categories.

Suppose we have a representation-theoretic structure $\mathcal Y_n$ in each $\Repb_{p_n}(S_{\nu_n})$ which can be constructed uniformly in an element-free way for every $n$. Then we can define the same structure $\mathcal Y$ in $\Rep(S_\nu)$ using the analogs of the same objects and maps. Since the definitions are the same, it would follow that $\mathcal Y = \prodF \mathcal Y_n$. Now one can try to transfer the properties of $\mathcal Y_n$ to $\mathcal Y$. For some it can be as easy as a direct application of \lthm. Others require quite a bit of technical work before one can do that. For some results of this type see \cite{kalinov2018finite,harman2020classification,etingof2020new}.

Oftentimes the structure $\mathcal Y$ might include some ind-objects of $\Rep(S_\nu)$. This will happen, for example, when we will try to define the rational Cherednik algebra  in $\Rep(S_\nu)$. Thus we  will deal with ind-objects in the ultraproduct setting in the next subsection.

\subsubsection{Ind-objects of $\Rep(S_\nu)$ as restricted ultraproducts}

In this section we are going to explain how ind-objects of $\Rep(S_\nu)$ can be obtained as restricted ultraproducts, thus extending Theorem \ref{ultdelthm} in a certain way. 

To do that, we will use the result of Construction \ref{indcons}.

\begin{prop} \label{inddelcons}
Suppose we have a sequence of representations  $M_n \in {\bf Rep}_{p_n}(S_{\nu_n})$, with fixed filtration by subrepresentations of finite length. i.e., we have $F^iM_n \in {\bf Rep}^f_{p_n}(S_{\nu_n})$ such that $\bigcup_{i \in \mathbb N}F^iM_n = M_n$. Also suppose that $\prodF^CF^iM_n \in {\rm Rep}(S_\nu)$. Then it follows that $M = \prodF^{C,r}M_n = \bigcup_{i \in \mathbb N}\prodF^C F^iM_n$ is an object of ${\rm IND}({\rm Rep}(S_\nu))$.\footnote{One can also define, through a more involved construction, the category $\IND(\Rep(S_\nu))$ as a subcategory of $\prodF \Repb_{p_n}(S_{\nu_n})$. Note that this subcategory will not be full. In this way one would also be able to consider $\prodF^C\bigcup_{i \in \mathbb N}F^iM_n$, i.e., take the ultraproduct directly. It can be shown that this would define the same object $M$.}
\end{prop}
\begin{proof}
This follows from Construction \ref{indconsult}.
\end{proof}

\begin{rem}
Note that, using Remark \ref{remindmor}, we conclude that if $M \in \IND(\Rep(S_\nu))$ has finite length, then for any $N \in \IND(\Rep(S_\nu))$ constructed via Proposition \ref{inddelcons}, we have:
$$
\Hom_{\IND(\Rep(S_\nu))}(M,N) = \bigcup_{j \in \mathbb N}\Hom_{\Rep(S_\nu)}(M,F^jN)=\bigcup_{j \in \mathbb N}\prodF\Hom_{\Repb_{p_n}(S_{\nu_n})}(M_{n},F^jN_{n})=
$$
$$
= {\prodF^r}\Hom_{\Repb_{p_n}(S_{\nu_n})}(M_{n},N_{n})  ,
$$
with the filtration arising from the filtration on $N$. 
\end{rem}

\subsection{Unital vector spaces and complex tensor powers}

Below we will use the notion of a  unital vector space. For details see \cite{etingof2014representation}.

\begin{def0}
A unital vector space $V$ is a vector space together with a unit, i.e., a distinguished non-zero vector denoted by $1 \in V$.
\end{def0}

In \cite{etingof2014representation} it is shown that given a finite dimensional unital vector space $V$, one can functorially define an ind-object $V^{\otimes \nu}\in \Rep(S_\nu)$. The idea behind this is that, although there is no way to algebraically define what $x^\nu$ is, there is, on the other hand, a way to define what $(1+x)^\nu$ is. Namely, $(1+x)^\nu:=\sum_{m\ge 0}\binom{\nu}{m}x^m$. 

We can also construct this object via an ultraproduct. Anyone not familiar with \cite{etingof2014representation} might regard this as definition for the purposes of this paper.

\begin{prop}\label{compow}
For a finite dimensional unital vector space $V$, the ind-object $V^{\otimes \nu}$ is given by:
$$
V^{\otimes \nu} = \prodF^{C,r} V^{\otimes \nu_n}. 
$$
\end{prop}
\begin{proof}
Using the notation of \cite{etingof2014representation}, we have:
$$
V^{\otimes \nu_n} = \bigoplus_{\lambda}S^{\lambda|_{\nu_n}}V \otimes X(\lambda|_{\nu_n}) ,
$$
where $S^{\lambda|_{\nu_n}}$ are the corresponding Schur functors. Thus we can define a filtration on each $V^{\otimes {\nu_n}}$ as 
$$
F^iV^{\otimes {\nu_n}} = \bigoplus_{|\lambda|\le i} S^{\lambda|_{\nu_n}}V \otimes X(\lambda|_{\nu_n}).
$$
Thus, taking the restricted ultraproduct with respect to this filtration, we obtain
$$
\prodF^{C,r} V^{\otimes \nu_n} = \bigcup_i \bigoplus_{|\lambda|\le i} \left( \prodF S^{\lambda|_{\nu_n}}V\right) \otimes \mathcal X(\lambda) = \bigoplus_{\lambda} S^{\lambda,\infty}V \otimes \mathcal X(\lambda)=V^{\otimes \nu},
$$
as needed. 

Note that we could have also used the filtration on $V^{\otimes \nu_n}$ induced by the filtration on $V$ given by $F^0V = \Bbbk \cdot 1$ and $F^1V=V$. I.e. the filtration, there the $i$-th term is spanned by all tensor monomials with no more than $i$ elements in the product that are not equal to $1$. Indeed this filtration is a sub-filtration of the filtration used above in the proof.

\end{proof}

\subsubsection{Symmetric powers of a unital algebra}\label{symalgsec}

In this section we will discuss a related construction in the case of the unital algebra. Here we will be concerned not with the tensor, but with symmetric powers of the unital vector space. Since the space of invariants of $V^{\otimes \nu}$ is an actual vector space, these objects will be usual vector spaces and not the objects of the Deligne category.

We will discuss the following class of algebras:
\begin{def0}
Consider $A$ -- a unital algebra. We will consider this algebra as a unital vector space with a unit given by the unit of the algebra. We call $A$ a filtered unital algebra if there is an ascending $\mathbb Z_{\ge 0}$-filtration by finite-dimensional subspaces such that $ \Bbbk \cdot 1\subset F^0A$. We will also suppose that such an algebra has a fixed vector space decomposition $A = \Bbbk \cdot 1 \oplus A'$.
\end{def0}

To make things clearer we will start with considering everything for transcendental $\nu$. I.e. we have $\Bbbk = \FQ$.

We would like to consider symmetric powers of a filtered unital algebra. I.e. we want to study the structure of $S^n(A)$. First of all note that this algebra admits a bifiltration.
\begin{def0} \label{SnAdef}
For a filtered unital algebra $A$, introduce a standard bifiltration of the algebra $S^n(A)$ in the following way. Consider $S^n(A)$ as $(A^{\otimes n})^{S_n}$. Introduce a bifiltration on $A^{\otimes n}$ via the following formulas for horizontal and vertical degrees:
$$
\deg_{h}(a_1 \otimes a_2 \otimes \dots \otimes a_n) = |\{i|a_i \notin  \Bbbk \cdot 1\}| \ ,
$$

$$
\deg_{v}(a_1 \otimes a_2 \otimes \dots \otimes a_n) = \sum_i\deg(a_i) \ .
$$
It is easy to see that this bifiltration restricts on the space of invariants of $S_n$.
\end{def0}
Now we can prove the following Proposition.
\begin{prop}
The associated graded algebra of the symmetric power $S^n(A)$ with respect to the horizontal filtration ${\rm gr}_h(S^n(A))$ is isomorphic to $\bigoplus_{i=0}^nS^i(A')$ as a vector space.
\end{prop}
\begin{proof}

Taking the associated graded with respect to the horizontal filtration allows us to use the standard splitting $A = \Bbbk \cdot 1 \oplus A'$. This allows us to view $A^{\otimes n}$ as $(\Bbbk \cdot 1 \oplus A')^{\otimes n}$. I.e. we have a  decomposition of ${\rm gr}_h(S^n(A))$ into a direct sum ${\rm gr}_h(S^n(A)) = \bigoplus_{i=0}^n {\rm gr}_h(S^k(A))_i$, where ${\rm gr}_h(S^n(A))_i$ consists of symmetric tensors, the tensor monomials of which have exactly $i$ components in $A'$ and the rest $n-i$ components are scalars $1$. I.e. we have:
$$
{\rm gr}_h(S^n(A))_i = [\bigoplus_{\sigma \in Sh(i,n-i)} C(\sigma(1))\otimes C(\sigma(2)) \otimes \dots \otimes C(\sigma(n))]^{S_n} \ ,
$$
where $C(1) = \dots = C(i) = A'$, $C(i+1) = \dots = C(n) = \Bbbk \cdot 1$ and $Sh(i,n-i)$ is the group of shuffles of $i$ and $n-i$. Hence:
$$
{\rm gr}_h(S^n(A))_i \simeq [A'^{\otimes i}]^S_i \otimes (\Bbbk \cdot 1)^{\otimes (n-i)} \simeq S^i(A') \ ,
$$
under the symmetrizing isomorphism. Hence we conclude that:
$$
{\rm gr}_h(S^n(A)) \simeq \bigoplus_{i=0}^n S^i(A') \ .
$$

Notice that the horizontal grading on the l.h.s. translates exactly into the grading by the degree of symmetric power on the r.h.s and the vertical filtration on l.h.s. translates into the filtration by the sum of degrees with respect to $A$ of the elements of the term in the symmetric product.
\end{proof}

Now we would like to consider an ultraproduct of such algebras:
\begin{def0}
For a filtered unital algebra $A$ over $\Bbbk =\FQ$, define $S^\nu(A)$ to be equal to an algebra $\prodF^r S^n(A)$ over $\mathbb C$, where the restricted ultraproduct is taken with respect to the total filtration of the bifiltered algebras. 
\end{def0}

Obviously this algebra inherits a bifiltration from $S^n(A)$. Thus we can consider ${\rm gr}_h(S^\nu(A))$. We can calculate this algebra with the help of the following Proposition.

\begin{prop}
We have a bifiltered vector space isomorphism between  
$$
{\rm gr}_h(S^\nu(A)) \simeq S^{\bullet}(A') \otimes_{\FQ} \mathbb C \ .
$$
\end{prop}
\begin{proof}
Indeed we have:
$$
{\rm gr}_h(S^\nu(A)) = \prodF^r {\rm gr}_h(S^n(A)) \simeq \bigoplus_{i=1}^{\infty} \prodF^r {\rm gr}_h(S^n(A))_i \ ,
$$
where the last restricted ultraproduct is taken with the respect to the filtration on ${\rm gr}_h(S^n(A))_i$ induced by the vertical filtration.

Now for each $n>i$, ${\rm gr}_h(S^n(A))_i$ has a filtered isomorphism with the same vector space $S^i(A')$. Hence $\prodF^r{\rm gr}_h(S^n(A))_i = S^i(A') \otimes_{\FQ}\mathbb C$. Thus we conclude:
$$
{\rm gr}_h(S^\nu(A)) \simeq S^{\bullet}(A') \otimes_{\FQ}\mathbb C \ .
$$
\end{proof}

To characterize this algebra more precisely we need to construct a certain map from $A$ to each $S^n(A)$.

\begin{prop} \label{deltaprop}
There is a map of Lie algebras $\delta_n: A \to S^n(A)$ (where the structure of the Lie algebra on both sides is given by the commutator) that sends $a \in A$ to $\sum_{i=1}^n a_i$, where $a_i = 1 \otimes \dots \otimes 1 \otimes a \otimes 1 \otimes \dots \otimes 1$, where $a$ is on the $i$-th place. This gives rise to an algebra map $\delta: A \otimes_{\FQ} \mathbb C\to S^\nu(A)$ that sends $1_A \mapsto \nu\cdot 1_{\mathcal A}$.
\end{prop}
\begin{proof}
Indeed $\delta_n$ is a well-defined map and it's a standard fact that it indeed gives us a map of Lie algebras. This map also respects the bifiltration if we consider the  horizontal filtration of $A$ to be given by $F_h^0A=\Bbbk \cdot 1$ and $F^1_hA = A$ and use the usual filtration on $A$ as the vertical one. Hence, taking an ultraproduct $\prodF \delta_n$ we obtain a well-defined map $\delta: A\otimes_{\FQ} \mathbb C \to S^{\nu}(A)$.

Now notice that under this map $\delta_n(1) = n \cdot 1 \otimes 1 \otimes \dots \otimes 1$.  The element $1 \otimes \dots \otimes 1 \in S^n(A)$ is the unity of this algebra. Thus $\prodF 1 \otimes \dots \otimes 1$ is the unity of $S^{\nu}(A)$. So, we conclude that:
$$
\delta(1_A) = \prodF \delta_n(1_A) = \prodF n \prodF 1 \otimes \dots \otimes 1 = \nu \cdot 1_{\mathcal A} \ .
$$
\end{proof}

This map allows us to define a map from $U(A)$: 
\begin{def0} \label{Deltadef}
Denote by $\Delta_k$ a map from the universal enveloping algebra $U(A)$ to $S^k(A)$ arising from the map $\delta_k$. 
\end{def0}

Now note that there is a bifiltration on $U(A)$ which comes from the bifiltration on $T^{\bullet}(A)$ arising from the bifiltration on $A$ and given by the same formulas as in Definition \ref{SnAdef}. With this filtration each $\Delta_k$ is a bifiltered morphism. This allows us to take their ultraproduct:
\begin{lemma}
The ultraproduct $\Delta = \prodF \Delta_k$ is a well defined bifiltered morphism from $U(A) \otimes_{\FQ} \mathbb C $ to $S^{\nu}(A)$.
\end{lemma}

Now we would like to prove that $\Delta$ is a surjective map.
\begin{lemma}  \label{Deltasur}
The map $\Delta$ is surjective.
\end{lemma}
\begin{proof}
It's enough to prove that all $\Delta_n$ are surjective and so it is enough to prove that $S^n(A)$ is generated by the image of $\delta_n$. We will do so by induction on the degree of the horizontal filtration. 

Now $F^{1,\bullet}S^n(A)$ is precisely the image of $\delta_n$ so the base of induction is clear.

Suppose that $F^{i-1,\bullet}S^n(A)$ is generated by the image of $\delta_n$. Suppose $f \in F^{i,\bullet}S^n(A)$. Now using the isomorphism of ${\rm gr}_h(S^n(A))$ with $\sum_{j=0}^n S^j(A')$, we may assume that $f = \widetilde{f}+ g$, where $\widetilde{f} = \sum \widetilde{f}_l$ and each $\widetilde{f}_l = a^{(l)}_1\otimes a^{(l)}_2 \otimes \dots \otimes a^{(l)}_{n} \otimes 1 \otimes \dots \otimes 1 + \textrm{shuffles}$, where each $a^{(l)}_i \in A'$ and $g \in F^{i-1,\bullet}S^n(A)$.

But now 
$$
h_l = \delta(a^{(l)}_1)\delta(a^{(l)}_2) \dots \delta(a^{(l)}_{n}) =
$$
$$
=a^{(l)}_1\otimes a^{(l)}_2 \otimes \dots \otimes a^{(l)}_{n} \otimes 1 \otimes \dots \otimes 1 + \ \textrm{shuffles} \  + \  \textrm{lower order terms in horizontal filtration}\ .
$$
Hence $f - \sum h_l \in F^{i-1,\bullet}S^k(A)$ and we are done.
\end{proof}

Now since we know that $\Delta(1_A) = \nu \cdot 1_{\mathcal A}$ it follows that $1_A - \nu \in \ker(\Delta)$. 
\begin{prop} \label{Deltaprop}
The map $\widetilde{\Delta}:U(A)/(1_A - \nu)\otimes_{\FQ}\mathbb C \to S^{\nu}(A)$ is a filtered algebra isomorphism. 
\end{prop}
\begin{proof}
We already know that this map is surjective. Now this map induces a graded map of the associated graded algebras with respect to the horizontal filtration. We know that ${\rm gr}_h(S^\nu(A)) = S^\bullet(A')$. Now ${\rm gr}_h(U(A)/(1_a - \nu))$ is isomorphic to ${\rm gr}(U(A')) \simeq S^\bullet(A')$. Hence, since the map is surjective, it also has to be injective. Hence $\widetilde{\Delta}$ is an isomorphism.
\end{proof}

\begin{rem} \label{remsymalg}
The same construction can be repeated in the case of algebraic $\nu$. In order to do so we should consider a lattice filtered unital algebra $A_{\mathbb Z}$ defined over $\mathbb Z$ and the sequence of algebras $A_n = A_{\mathbb Z} \otimes_{\mathbb Z} \Fpn$.

In this case as we know $\nu_n < p_n$,  all of the constructions which use the isomorphisms related to symmetric invariants work in the same way and we can still define the $\nu$-symmetric power as $S^\nu(A) = \prodF^r S^{\nu_n}(A_n)$. Everything else can be repeated and we obtain a similar isomorphism $\widetilde{\Delta}: U(A_{\mathbb Z}\otimes_{\mathbb Z} \mathbb C)/(1_A - \nu) \to S^\nu(A)$.
\end{rem}

\section{Extended Cherednik algebras}

\subsection{Definition and basic facts}

In this section we will introduce the notion of the extended Cherednik algebra and will discuss its basic properties. This algebra was introduced in \cite{etingof2020representations}, see this paper for more information regarding it.  Everywhere we suppose that $\chct (\Bbbk) > n$.

\begin{def0} [Definition 2.4 in \cite{etingof2020representations}] \label{extcherfin}
For $t,k  \in \Bbbk $ and $n,r\in \mathbb Z_{>0}$ define the extended Cherednik algebra $H_{t,k}(n,r)$ to be a quotient of the semi-direct product:
$$
\Bbbk[S_n] \ltimes [\Bbbk\langle x_1,\dots, x_n,y_1,\dots,y_n\rangle \otimes (\End(\Bbbk^r))^{\otimes n}] \ ,
$$
where $S_n$ acts by permuting $x_i, y_i$ and the copies of $\End(\Bbbk^r)$. The quotient is taken by the ideal generated by the following relations:
\begin{gather*}
    [x_i,x_j] = 0 \ ,\ [y_i,y_j] = 0 \ , \\
    [y_i,x_j] = \delta_{ij}(t-k\sum_{m \ne i}s_{im}\sigma_{im}) + (1-\delta_{ij})ks_{ij}\sigma_{ij} \ ,
\end{gather*}
where $s_{ij}$ are the transpositions from $S_n$ viewed as elements of $\Bbbk [S_n]$ and $\sigma_{ij}$ is the following element of
$\enalg^{\otimes n}$:
$$
\sigma_{ij} = \sum_{\alpha,\beta} (E_{\alpha\beta})_i(E_{\beta\alpha})_j \ .
$$
Here by $(g)_i$ for $g \in \End(\Bbbk^r)$ we denote an element of $\End(\Bbbk^r)^{\otimes n}$ which is equal to $1 \otimes \dots \otimes g \otimes \dots \otimes 1$ with $g$ on the $i$-th place\footnote{We use the extra brackets around $g$ here, since in what follows we will consider the cases where the elements of $\enalg$ we are going to use are equal to elementary matrices.}. Notice that $\sigma_{ij}$ as an operator acting on $(\Bbbk^r)^{\otimes n}$ is exactly the operator which transposes the $i$-th and $j$-th spaces.
\end{def0}

Obviously for $r=1$ the algebra $H_{t,k}(n,1)$ is just the usual rational Cherednik algebra of type $A_{n-1}$. Now there is also an analogue of the polynomial representation for $H_{t,k}(n,r)$. 
\begin{prop} [Proposition 2.7 in \cite{etingof2020representations}]\label{extcherstrep}
Consider the vector space \linebreak $V(n,r) = \Bbbk[x_1,\dots,x_n]\otimes (\Bbbk^r)^{\otimes n}$. It has a natural action of $H_{t,k}(n,r)$ given by the following formulas:
\begin{gather*}
    x_i \mapsto x_i \cdot \otimes 1 \ , \ s_{ij} \mapsto s^{x}_{ij}\otimes \sigma_{ij} \ , (g)_i \mapsto 1 \otimes (g)_i \\
    y_i \mapsto \partial_i \otimes 1 - k\sum_{j \ne i}\frac{s^{x}_{ij} \otimes 1}{x_i-x_j} \ ,
\end{gather*}
where $s^x_{ij}$ is the transposition acting on $\Bbbk[x_1,\dots,x_n]$. 
\end{prop}
\begin{proof}
Notice that $s_{ij}^x = s_{ij}\sigma_{ij}$ in this representation. Using this it is easy to see that all the operators satisfy the required relations.
\end{proof}

\begin{cor} [Proposition 2.8 in \cite{etingof2020representations}]
The algebra $H_{t,k}(n,r)$ enjoys the PBW-property in the sense that the multiplication map 
$$
\Bbbk[x_1,\dots,x_n] \otimes [\Bbbk [S_n]\otimes (\End(\Bbbk^r))^{\otimes n}] \otimes \Bbbk[y_1,\dots,y_n] \to H_{t,k}(n,r)
$$
is an isomorphism. Moreover the multiplication maps for any other ordering of tensor multiples are also isomorphisms. 
\end{cor}
\begin{proof}
This follows from the fact that the polynomial representation introduced above is faithful and the image of $H_{t,k}(n,r)$ in  $\End(V(n,r))$ is the subalgebra of \linebreak $ \Bbbk[S_n] \ltimes(\mathcal D^{reg}(\mathbb A^n) \otimes \End(\Bbbk^r)^{\otimes n})$, where $\mathcal  D^{reg}(\mathbb A^n)$ is the algebra of differential operators on the regular locus of $\mathbb A^n$.
\end{proof}

We can also define the spherical subalgebra of $H_{t,k}(n,r)$.
\begin{def0}
For $t,k  \in \Bbbk $ and $n,r\in \mathbb Z_{>0}$ define the spherical subalgebra of the extended Cherednik algebra to be $B_{t,k}(n,r) = \bold{e}H_{t,c}(n,r)\bold{e}$, where $\bold{e}$ is a symmetrizing element $\bold{e} = \frac{1}{n!}\sum_{s \in S_n}s$.
\end{def0}

There is a natural vector space bifiltration on $H_{t,k}(n,r)$. 

\begin{def0}Assign to an element $\prod_i x_i^{n_i}s\bigotimes_i (g_i)_i \prod_i y_i^{m_i} \in H_{t,k}(n,r)$ the following bidegree. Denote by $H = |\{i \in \{1,\dots,n\}\ |\  n_i = 0, \ g_i \notin \Bbbk \cdot \Id_{\Bbbk^r}, \ m_i = 0 \}|$, and by $V = \sum_{i} (n_i+m_i)$. Then $\deg(\prod_i x_i^{n_i}s\bigotimes_i (g_i)_i \prod_i y_i^{m_i}) = (n-H,V)$. Define the bifiltration on $H_{t,k}(n,r)$ using this formula.
\end{def0}
I.e. the horizontal degree tells us how many indices actually appear in the monomial, and the vertical degree is the total polynomial degree of the monomial. Note that this is not an algebra bifiltration. The same vector space bifiltration restricts to the spherical subalgebra.

However note that the associated graded of $H_{t,k}(n,r)$ with respect to the vertical filtration is simply ${\rm gr}_v(H_{t,k}(n,r)) \simeq \Bbbk [S_n] \ltimes (\Bbbk [x_1,\dots,x_n,y_1, \dots,y_n] \otimes \End(\Bbbk^r)^{\otimes n}) $. Now the vector space bifiltration of $H_{t,k}(n,r)$ restricts to ${\rm gr}_v(H_{t,k}(n,r))$ and makes it a bifiltered algebra.

Moreover, the associated graded of the spherical subalgebra $B_{t,k}(n,r)$ is given by a similar formula  ${\rm gr}_v(B_{t,k}(n,r)) \simeq (\Bbbk [x_1,\dots,x_n,y_1, \dots,y_n] \otimes (\End(\Bbbk^r))^{\otimes n})^{S_n}$. And again this associated graded is a bifiltered algebra.

But now this is simply ${\rm gr}_v(B_{t,k}(n,r)) \simeq S^n(\End(\Bbbk^r)[x,y])$. And the bifiltration on ${\rm gr}_v(B_{t,k}(n,r))$  coincides exactly with the standard bifiltration of $S^n(\End(\Bbbk^r)[x,y])$ arising from the fact that $\End(\Bbbk^r)[x,y]$ is a filtered unital algebra with the filtration given by the total degree of the polynomial. I.e. we are now in the setting of Section \ref{symalgsec} with $A = \End(\Bbbk^r)[x,y]$. We will use the results of that Section below to construct a generating set of $B_{t,k}(n,r)$ and then later, when we take the ultraproduct of $B_{t,k}(n,r)$ to obtain the DDCA.

\begin{rem} \label{sphfinrem}
As the final remark of this section note that we can also construct the spherical subalgebra using the induction functor in the following way:
$$
B_{t,k}(n,r) = \Hom_{S_k}(\Bbbk,H_{t,k}(n,r)\bold{e}) = \End_{H_{t,k}(n,r)}(H_{t,k}(n,r)\bold{e}) = \End_{H_{t,k}(n,r)}(\Ind_{S_k}^{H_{t,k}(n,r)}(\Bbbk)) \ .
$$
\end{rem}

\subsection{Generating set of $ B_{t,k}(n,r)$} \label{gensetsect}

In this section we would like to present a way to construct a generating set for $B_{t,k}(n,r)$.

Pick a basis in $\End(\Bbbk^r)$, which contains $\Id_{\Bbbk^r}$ as an element. Let us denote this basis by $\alpha_i \in \End(\Bbbk^r)$ with $i$ going from $1$ to $r^2$ and $\alpha_1 = \Id_{\Bbbk^r}$. Now we can define the following elements in $B_{t,k}(n,r)$.

\begin{def0}
Define the elements $T_{r,q,n}(g)$, for $g \in \enalg$ by the following formula (here $L = r+q$):
$$
\sum_{r,q \ge 0, r+q = L}T_{r,q,n}(g)\frac{u^r}{r!}\frac{v^q}{q!}= \sum_{i=1}^n(g)_i\frac{(ux_i +vy_i)^L }{L!}\bold{e} ,
$$
where $u,v$ are formal variables. 

They are defined for $\chct(\Bbbk) > r+q$ or zero characteristic.
\end{def0}

More explicitly,  $T_{r,q,n}(g)$ is proportional to the sum of all shuffles of $r$ copies of $x_i$ and $q$ copies of $y_i$ multiplied by $(g)_i$ and summed over all $i$ from $1$ to $n$.

Note that the highest term of $T_{r,q,n}$ with respect to the vertical filtration is:
$$
T_{r,q,n}(g) = \sum_i (g)_i x_i^ry_i^q  + \textrm{lower order terms} \ .
$$

Suppose that $\chct(\Bbbk)=0$. Now note that we have established that we have an isomorphism ${\rm gr}_v(B_{t,k}(n,r)) \simeq S^n(\enalg[x,y])$. From Proposition \ref{deltaprop} we know that there is a map $\delta_n: \enalg[x,y] \to S^n(\enalg[x,y])$. Note that under this map \linebreak $\delta_n(g \cdot x^ry^q) = \sum_i (g)_i x_i^r y_i^q$. I.e. the images of $T_{r,q,n}(g)$ in the associated graded span exactly the image of the map $\delta_n$. More precisely it is enough to consider all $T_{r,q,n}(\alpha_l)$ for $r,q \in \mathbb Z_{\ge 0}$ and $\alpha_l \in \{1,\dots,r^2\}$ to span this image, since $\alpha_l\cdot x^ry^q$ for all such $r,q,l$ are the basis of $\enalg[x,y]$.

Now we can define something like the shuffled products of the above elements. 

\begin{def0}
Denote by $\bold{m}$ the collection of integers $m_{r,q,l}$, for $r,q \in \mathbb Z_{\ge 0}$ and \linebreak $l \in \{1,\dots,r^2\}$. Denote  $|\bold{m}| =\sum_{r,q,l}m_{r,q,l}$ and $w(\bold{m}) = \sum_{r,q,l}(r+q)m_{r,q,l}$. We define $T_n(\bold{m})$ with $M = |\bold{m}|$, by the following formula:
$$
\sum_{\bold {m}, |\bold{m}| = M}T_n(\bold{m}) \prod_{r,q,l}\frac{z_{r,q,l}^{m_{r,q,l}}}{m_{r,q,l}!} = \frac{(\sum_{r,q,l} z_{r,q,l}T_{r,q,n}(\alpha_l))^M}{M!} \ ,
$$
here $z_{r,q,l}$ are formal variables.

Note that these elements are defined for $w(\bold{m}) < \chct(\Bbbk)$ or zero characteristic. 
\end{def0}

Note that with respect to the total filtration 
$$T_n(\bold{m}) = \prod_{r,q,l} T_{r,q,n}(\alpha_l))^{m_{r,q,l}} + \textrm{lower order terms}  \ . 
$$

Suppose we are in the case $\chct(\Bbbk) = 0$.

Since $T_{r,q,n}(\alpha_l)$ span the image of $\delta_n$ in the associated graded algebra with respect to the vertical filtration, it follows that $T_n(\bold{m})$ span everything which is generated by $\delta_n$ inside ${\rm gr}_v(B_{t,k}(n,r))$. If we use the notation of Definition \ref{Deltadef}, we can state this by saying that the images of $T_n(\bold{m})$ in the associated graded span the image of $\Delta_n$. But from Proposition \ref{Deltasur} we know that this image covers the whole algebra. So the following Proposition follows.

\begin{prop}
Suppose $\chct(\Bbbk)=0$. The elements $T_n(\bold{m})$ for all choices of $\bold{m}$ form a generating set of $B_{t,k}(n,r)$.
\end{prop}

\subsection{Extended Cherednik algebras in complex rank}

In this section we will explain how to work with the extended Cherednik algebras in the complex rank. First we will define a category of representations of $H_{t,k}(\nu,r)$.

In order to do this we need to explain a few things about the central elements in $\Rep(S_\nu)$. This will build on the discussion around Definition \ref{centraldef}.

\begin{cons}
Let us define the action of the  central element $\Omega$ on objects of $\Rep(S_\nu)$. Consider $E_2 \subset \Bbbk[S_\nu]$ as defined in \cite{etingof2014representation}. This is the interpolation of the subspaces spanned by transpositions in the group algebra. Then we have a map \linebreak $\Delta_{E_2}: E_2 \to E_2 \otimes E_2$ that interpolates the usual coproduct map $\Delta(s_{ij}) =s_{ij}\otimes s_{ij} $. Also we have a map $\omega:\Bbbk \to E_2$ interpolating the central element inclusion map $1 \mapsto \sum s_{ij} = \Omega$. We also automatically have an action map $a_{E_2}:E_2 \otimes V\to V$ for any object $V \in \Rep(S_\nu)$. Thus we get the alternative way to define the map $\Omega:V \to V$ given by the identity functor endomorphism. More precisely, this map is given by $a_{E_2}\circ (\omega \otimes 1)$. 
\end{cons}

For our purposes we need to slightly upgrade this central element. 
\begin{cons} \label{extcentel}
Note that there is a map $i_{E_2}: E_2 \to \mathfrak h \otimes \mathfrak h$, which interpolates the map $s_{ij} \mapsto \frac{x_i \otimes x_j + x_j \otimes x_i}{2}$. Also consider a map $coev_{\End(\Bbbk^r)}:\Bbbk \to \End(\Bbbk^r)\otimes \End(\Bbbk^r)$ (i.e. we have $1 \mapsto \sum_{i,j} E_{ij}\otimes E_{ji}$). Now we can construct $\omega_{\End(\Bbbk^r)}$ as follows:
$$
\omega_{\End(\Bbbk^r)} = (1 \otimes \sigma_{\End(\Bbbk^r), \mathfrak h}\otimes 1)\circ(1\otimes i_{E_2}\otimes 1) \circ(1 \otimes \Delta_{E_2}) \circ (1 \otimes \omega)\circ tw_{\End(\Bbbk^r)} \ ,
$$
which takes $\Bbbk \to \End(\Bbbk^r) \otimes \mathfrak h \otimes \End(\Bbbk^r) \otimes \mathfrak h \otimes E_2$.

Now suppose $V$ is an object of $\Rep(S_\nu)$, with a fixed map $\alpha: \End(\Bbbk^r) \otimes \mathfrak h \otimes V \to V$. Then we can define $\Omega_{\End(\Bbbk^r)}: V \to V$ as 
$$
\Omega_{\End(\Bbbk^r)} =\alpha \circ(1\otimes \alpha )\circ(1 \otimes a_{E_2})\circ(\omega_{\End(\Bbbk^r)} \otimes 1) \ .
$$
So we have another "central element" for special objects of $\Rep(S_\nu)$. 
\end{cons}

\begin{def0} \label{ultcherdef}
The category $\Rep(H_{t,k}(\nu,r))$ is defined as follows. The objects are given by  triples $(M,x,y,\alpha)$, where $M$ is an ind-object of $\Rep(S_\nu)$, $x$ is a map \linebreak  $x: \mathfrak h^* \otimes M \to M$, $y$ a map $y: \mathfrak h \otimes M \to M$ and $\alpha$ is a map $\alpha:(\End(\Bbbk^r) \otimes \mathfrak h)\otimes M \to M$, all of which are morphisms in $\IND(\Rep(S_{\nu}))$. They are required to satisfy the following conditions:
$$
x \circ (1\otimes x) -x \circ (1\otimes x) \circ (\sigma \otimes 1) = 0  , 
$$
as a map from $\mathfrak h^* \otimes \mathfrak h^* \otimes M$ to $M$;
$$
y \circ (1\otimes y) -y \circ (1\otimes y) \circ (\sigma \otimes 1) = 0  , 
$$
as a map from $\mathfrak h \otimes \mathfrak h \otimes M$ to $M$;
\begin{gather*}
\alpha \circ (1\otimes \alpha) - \alpha \circ (1\otimes \alpha) \circ (\sigma_{\enalg\otimes \mathfrak h}\otimes 1) = \\
= \alpha \circ (\mu_{\enalg}\otimes 1 - [\mu_{\enalg}\otimes 1]\circ [\sigma_{\enalg}\otimes 1])   \circ(1 \otimes \pi_{\rm{diag}} \otimes 1)\circ(1\otimes \sigma_{\mathfrak h, \enalg}\otimes 1),
\end{gather*}
as a map from $\End(\Bbbk^r) \otimes \mathfrak h \otimes \End(\Bbbk^r) \otimes \mathfrak h \otimes M \to M$, where $\mu_{\End(\Bbbk^r)}$ is multiplication in $\End(\Bbbk^r)$ and $\pi_{\rm{diag}}:\mathfrak h\otimes \mathfrak h \to \mathfrak h$ is the interpolation of the projection $x_i \otimes x_j \mapsto \delta_{ij} x_i$;
$$
\alpha\circ(\iota_{\End(\Bbbk^r)} \otimes 1) - 1 \otimes \Tr_{\mathfrak{h}}\otimes 1 = 0 , 
$$
as a map $\Bbbk \otimes \mathfrak h \otimes M \to M$, where $\iota_{\End(\Bbbk^r)}$ is the unit map of $\End(\Bbbk^r)$ and $\Tr_{\mathfrak{h}}: \mathfrak h \to \Bbbk$ is the trace, the interpolation of the map $x_i \mapsto 1$;
$$
y \circ (1\otimes x) -x \circ (1\otimes y) \circ (\sigma \otimes 1) = t \cdot {\rm ev}_{\mathfrak h} \otimes 1 - k \cdot ({\rm ev}_{\mathfrak h}\otimes 1) \circ( \Omega_{\End(\Bbbk^r)}^{3} - \Omega_{\End(\Bbbk^r)}^{1,3}) , 
$$
as a map  $\mathfrak h \otimes \mathfrak h^* \otimes M$ to $M$, 
where $\Omega_{\End(\Bbbk^r)}$ is a central element from Construction \ref{extcentel}, and indices indicate the spaces on which $\Omega_{\End(\Bbbk^r)}$ acts in the tensor product $\mathfrak h \otimes \mathfrak h^* \otimes M$.

The morphisms of $\Rep(H_{t,k}(\nu))$ are the morphisms of $\IND(\Rep(S_\nu))$ which commute with the action-maps $x$,$y$ and $\alpha$.
\end{def0}

\begin{rem}
Some comments are in order to explain why this is indeed the correct generalization of Definition \ref{extcherfin}. To see that one needs to understand that  Definition \ref{ultcherdef} above, if used in the finite rank, gives us the usual category of representations of the extended Cherednik algebra $H_{t,k}(n,r)$. Indeed, note that since $M$ is already an object of the category of representations of symmetric group, we do not need to define its action. Now maps $x$ and $y$ determines the action of elements $x_i$ and $y_i$. The map $\alpha$ determines the action of elements $(g)_i$. The first two formulas tell us that $x_i$ commute with each other and so also $y_i$. The third formula gives us the commutation relation between $(g)_i$ and $(h)_j$ (i.e. $[(g)_i,(h)_j] = \delta_{ij}([g,h])_i$). The fourth tells us that all $(1)_i$ act trivially. And, finally, the fifth formula, if expanded, gives us the correct commutation relation between $x_i$ and $y_j$.
\end{rem}

Now we would like to show how we can construct some of the objects of  the category $\Rep(H_{t,k}(\nu,r))$ as ultraproducts.

\begin{rem}
Below we will denote by $t_n,k_n$ the elements of $\Fpn$  such that $\prodF t_n = t$ and $\prodF k_n = k$ under the fixed isomorphism of $\prodF \Fpn \simeq \mathbb C$. We will use a similar notation for all other parameters of algebras used in the paper.
\end{rem}

\begin{lemma} \label{thmcherault}
Suppose $M_n$ is a sequence of objects of ${\bf Rep}_{p_n}(H_{{t_n},{k_n}}(\nu_n,r))$ such that their (restricted) ultraproduct as objects of ${\bf Rep}_{p_n}(S_{\nu_n})$ lies in ${\rm IND}({\rm Rep}(S_\nu))$. Suppose $x_n$, $y_n$ and $\alpha_n$ are the maps which define the action of generators of the corresponding Cherednik algebra on $M_n$. Then $(\prodF^{C,r}M_n, \prodF x_n, \prodF y_n,\prodF \alpha_n)$ defines an object of ${\rm Rep}(H_{t,k}(\nu,r))$.
\end{lemma}
\begin{proof}
It's easy to see that the data $(\prodF^{C,r}M_n, \prodF x_n, \prodF y_n,\prodF \alpha_n)$ is well defined. Since $x_n$, $y_n$ and $\alpha_n$ satisfy the same conditions in finite rank and complex rank it follows that by \lthm${}$ this is indeed an object of $\Rep(H_{t,k}(\nu,r))$.
\end{proof}

Now we would like to construct an interpolation of the functors $\Ind_{S_{\nu_n}}^{  H_{t_n,k_n}(\nu_n,r)}$. It is possible to construct the full functor as an ultraproduct, but this functor would a priori have $\prodF \Repb_{p_n}(H_{t_n,k_n}(\nu_n,r))$ as its target category, so we would need to explain why the functor really gives us objects of $\Rep(H_{t,k}(\nu,r))$. Instead we will construct this functor explicitly, which will also show that it agrees with the ultraproduct functor when applied to objects of $\Rep(S_\nu)$.

The idea is, following the PBW theorem, to think about ``$ H_{t,k}(\nu,r)$" as ``the direct sum $\bigoplus_{i,j \ge 0}S^i(\mathfrak h^*) \otimes S^j(\mathfrak h) \otimes  (\End(\Bbbk^r))^{\otimes \nu} \otimes \mathbb C[S_\nu]$" and take the tensor product with $V \in \Rep(S_\nu)$ ``over $\mathbb C[S_\nu]$".

Before the actual construction we need to note several things.
\begin{cons}
Denote $A = \End(\Bbbk^r)$. First, since $A$ is a unital algebra with the standard filtration $F^0A = \Bbbk \cdot 1$ and $F^1A = A$, we have an induced filtration on $A^{\otimes \nu}$. Note that $A^{\otimes \nu}$ as an algebra is generated by its first filtration component $F^1A^{\otimes \nu}$. This component itself is actually a subobject of $\mathfrak h \otimes A$, more precisely to obtain it we need to throw out a subobject $\mathcal X((1)) \otimes F^0A$ from $\mathfrak h \otimes A$ (note $\mathfrak h = \mathcal X((1)) \oplus \Bbbk$). It follows that there are maps $i_{l,A}: F^lA^{\otimes \nu} \to (\mathfrak h\otimes A)^{\otimes l}$ and $\pi_{l,A}: (\mathfrak h\otimes A)^{\otimes l} \to F^lA^{\otimes \nu}$. Let us denote the multiplication map $\pi_{l+1,A} \circ 1 \otimes i_{l,A}: (\mathfrak h \otimes A) \otimes F^lA^{\otimes \nu} \to F^{l+1}A^{\otimes \nu}$ by $\mu_{l,A}$.

Also note that $S^{i+1}(\mathfrak h)$ is isomorphic to a direct summand of $\mathfrak h \otimes S^{i}(\mathfrak h)$, let's denote the corresponding inclusion and projection as $\iota_{i+1, y}$ and $\pi_{i+1,y}$ respectively. The same is true for $\mathfrak h^*$, the corresponding morphisms are $\iota_{i+1,x}$ and $\pi_{i+1,x}$.
\end{cons}

With this we can proceed to construct the induction functor.

\begin{cons} \label{inducdef}

 For an object $V \in \Rep(S_\nu)$, consider an ind-object $I_V= \oplus_{i,j\ge 0}I_{i,j}$, where $I_{i,j} = S^i(\mathfrak h^*) \otimes S^j(\mathfrak h ) \otimes A^{\otimes \nu} \otimes V$, and maps $x_V: \mathfrak h^* \otimes I_V \to I_V$, $y_V: \mathfrak h \otimes I_V \to I_V$ and $\alpha_V: (\mathfrak h \otimes A) \otimes I_V \to I_V$, which are defined as follows.

First let us define $\alpha_V|_{I_{i,j}}:(\mathfrak h \otimes A) \otimes I_{i,j} \to I_{i,j}$. We will do so by considering the action of this map on each filtration component $F^lI_{i,j} = S^i(\mathfrak h^*) \otimes S^j(\mathfrak h ) \otimes F^lA^{\otimes \nu} \otimes V$. Now we can define the action of $\alpha_V|_{F^lI_{i,j}}:(\mathfrak h \otimes A) \otimes F^lI_{i,j} \to F^{l+1}I_{i,j}$ to be equal to $(1 \otimes \mu_{l,A} \otimes 1) \circ(\sigma_{\mathfrak h \otimes A,S^i(\mathfrak h^*) \otimes S^j(\mathfrak h )}\otimes 1)$.

Now define $x_V|_{I_{i,j}}: \mathfrak h^*\otimes I_{i,j} \to I_{i+1,j}$ to be equal to $\pi_{i+1,x} \otimes 1$ for all $i,j$. Also define $y_V|_{I_{0,j}}: \mathfrak h\otimes I_{0,j} \to I_{0,j+1}$ as $\pi_{j+1,y}\otimes 1$. And lastly we define $y_V|_{I_{i,j}}: \mathfrak h \otimes I_{i,j} \to I_{i,j+1} \oplus I_{i-1,j}$ by induction in $i$ as:
\scriptsize
$$
\left[(x \otimes 1) \circ(1 \otimes y \otimes 1) \circ (\sigma_{\mathfrak h, \mathfrak{h}^*} \otimes 1) + t \cdot {\rm ev}_{\mathfrak h} \otimes 1 - k\cdot ({\rm ev}_{\mathfrak h } \otimes 1) \circ (\Omega_A^{I_{i-1,j}} - \Omega_A^{\mathfrak h, I_{i-1,j}})\right]\circ(1 \otimes \iota_{i,x} \otimes 1)  .
$$ 
\normalsize
\end{cons}

Now we would like to show that this defines an object of $\Rep( H_{t,k}(\nu,r))$. Indeed:
\begin{lemma} \label{indlemma1}
In the notations of Construction \ref{inducdef}, the tuple $(I_V,x_V,y_V,\alpha_V)$ defines an object of $\Rep( H_{t,k}(\nu,r))$.
\end{lemma}
\begin{proof}
Indeed, the first two formulas of Definition \ref{ultcherdef} are satisfied by the properties of symmetric powers, and we defined the action of $y_V$ by induction in such a way that the last equation is also satisfied. The equations for $\alpha_V$ are satisfied in a straightforward way.

Another way to see that is to note that in the finite rank case this construction amounts to $ H_{{t_n},{k_n}}(\nu_n) \otimes_{S_{\nu_n}} V_n$, and so by \lthm, we do get a correct structure of an \linebreak ``$ H_{t,k}(\nu)$-module".
\end{proof}

Now we need to construct the action of the induction functor on  morphisms.
\begin{cons} \label{indchermordef}
In the notation of Construction \ref{inducdef}, given a morphism \linebreak $\phi: V \to U$, define a  morphism $I_\phi: I_V\to I_U$ in the following way:
$$
(I_\phi)|_{S^i(\mathfrak h^*) \otimes S^j(\mathfrak h) \otimes A^{\otimes \nu} \otimes V}:=1 \otimes \phi \ .
$$
\end{cons}
\begin{lemma} \label{indlemma2}
In the notation of Constructions \ref{inducdef} and \ref{indchermordef}, $I_\phi$ is a morphism in ${\rm Rep}( H_{t,k}(\nu,r))$.
\end{lemma}
\begin{proof}
This is easy to see both straight from the definition, or by the ultraproduct argument, since in finite rank this defines an actual $ H_{t_n,k_n}(\nu_n,r)$-module morphism.
\end{proof}

Now we can define the actual functor:
\begin{def0} \label{indfunctor}
Define a functor $\Ind_{S_\nu}^{ H_{t,k}(\nu,r)}: {\rm Rep}(S_\nu) \to \Rep(H_{t,k}(\nu,r))$ in the following way. On objects it takes $V$ to the triple $(I_V,x_V,y_V,\alpha_V)$ from Construction \ref{inducdef}. And on morphisms it takes $\phi: V \to U$ to $I_\phi$ from Construction \ref{indchermordef}. This is a well defined functor by Lemmas \ref{indlemma1} and \ref{indlemma2}.
\end{def0}
The next Corollary follows by construction and the above lemmas:
\begin{cor} \label{indcor}
For any object $V \in {\rm Rep}(S_\nu)$ such that $V = \prodF V_n$ we have:
$$
{\rm Ind}_{S_\nu}^{ H_{t,k}(\nu,r)}V = \prodF^{C,r} {\rm Ind}_{S_{\nu_n}}^{ H_{{t_n},{k_n}}(\nu_n,r)}V_n  ,
$$
where the filtration on ${\rm Ind}_{S_{\nu_n}}^{ H_{{t_n},{k_n}}(\nu_n,r)}V_n$ is obtained from the vector space bifiltration of $H_{{t_n},{k_n}}(\nu_n,r)$ (which can be seen to be $S_{\nu_n}$-invariant).
\end{cor}
\section{DDCA in complex rank}

In this section we will define and study the Deformed Double Current Algebra of rank $r$. Our train of thought will resemble that of \cite{etingof2020new} where the case of $r=1$ was studied.

\subsection{Definition}

Now we can define the Deformed Double Current algebra of rank $r$. We will construct it as an algebra of endomorphisms in $\Rep(H_{t,k}(\nu,r))$.

\begin{def0}
For $r \in \mathbb Z_{> 0}$, $\nu \in \mathbb C \backslash \mathbb Z_{\ge 0}$ and $t,k \in \mathbb C$, define the DDCA of rank $r$, denoted $\widetilde{\mathcal D}_{t,k,\nu}(r)$, as:
$$
\widetilde{\mathcal D}_{t,k,\nu}(r) = \End_{\Rep(H_{t,k}(\nu,r))}(\Ind_{S_\nu}^{H_{t,k}(\nu,r)}(\mathbb C)) \ .
$$
\end{def0}

This is obviously an interpolation of Remark \ref{sphfinrem}. This can be made precise using Corollary \ref{indcor}:
\begin{prop}
The algebra $\widetilde{\mathcal D}_{t,k,\nu}(r)$ is equal to a restricted ultraproduct of $B_{t_n,k_n}(\nu_n,r)$ with respect to the total filtration.
\end{prop}
\begin{proof}
Indeed since 
$$
\widetilde{\mathcal D}_{t,k,\nu}(r)= \End_{\Rep(H_{t,k}(\nu,r))}(\Ind_{S_\nu}^{H_{t,k}(\nu,r)}(\mathbb C)) = \Hom_{\Rep(S_\nu)}(\mathbb C,\Ind_{S_\nu}^{H_{t,k}(\nu,r)}(\mathbb C) )
$$
by Corollary \ref{indcor}, it follows that:
$$
\widetilde{\mathcal D}_{t,k,\nu}(r)= \prodF^r \Hom_{\Repb_{p_n}}(\Fpn, \Ind_{S_{\nu_n}}^{H_{t_n,k_n}(\nu_n,r)}(\Fpn)) = \prodF^r B_{t_n,k_n}(\nu_n,r) \ .
$$
\end{proof}

\begin{rem}
Note that $\widetilde{\mathcal D}_{t,k,\nu}(r)$ has a vector space bifiltration which it inherits through the ultraproduct construction from the similar filtrations on $B_{t_n,k_n}(\nu_n,r)$.
\end{rem}

\begin{rem}
Also note that all of the above can be repeated \textit{verbatim} for the case of $\Rep^{\rm{ext}}(S_\nu)$ from Definition \ref{extdef}. In this case we obtain the algebra $\widetilde{\mathcal D}^{\rm{ext}}_{t,k,\nu}(r)$ over $\Cext$.
\end{rem}

\subsection{Basis of DDCA}

In this subsection we would like to generalize the elements from the Section \ref{gensetsect} to the DDCA.

\begin{cons}
Consider elements $T_{\nu_n}(\bold{m})$ of $B_{t_n,k_n}(\nu_n,r)$ as maps belonging to $\Hom_{\Repb_{p_n}(S_{\nu_n})}(\Fpn, H_{t_n,k_n}(\nu_n,r)\bold{e})$. Since these elements are defined for big enough characteristic, they are defined \fan. And since their degree as maps is bounded, it follows that the ultraproduct $T(\bold{m}) = \prodF T_{\nu_n}(\bold{m})$ is a well-defined element of $\widetilde{\mathcal D}_{t,k,\nu}(r)$. And the same hold for $\widetilde{\mathcal D}^{\rm{ext}}_{t,k,\nu}(r)$.
\end{cons}

Now let us consider the associated graded of $\widetilde{\mathcal D}_{t,k,\nu}(r)$ with respect to the vertical filtration.

\begin{prop}
The associated graded algebra ${\rm gr}_v(\widetilde{\mathcal D}_{t,k,\nu}(r))$ is isomorphic to \linebreak $S^{\nu}(\enalg[x,y])$ as a bifiltered algebra.
\end{prop}
\begin{proof}
We have ${\rm gr}_v(\widetilde{\mathcal D}_{t,k,\nu}(r)) = \prodF^r {\rm gr}_v(B_{t_n,k_n}(\nu_n,r))$. But since we know what the associated graded of the spherical subalgebra is, it follows that:
$$
{\rm gr}_v(\widetilde{\mathcal D}_{t,k,\nu}(r)) = \prodF^r S^{\nu_n}(\End(\Fpn^r)[x,y]) = S^\nu(\enalg[x,y]) \ .
$$
\end{proof}

By Remark \ref{remsymalg} it follows that there is an isomorphism:
$$
\widetilde{\Delta}: U(\enalg[x,y])/(1_{\enalg}- \nu) \simeq {\rm gr}_v(\widetilde{\mathcal D}_{t,k,\nu}(r)) \ .
$$

From this we can derive the following statement about the basis of $\widetilde{\mathcal D}_{t,k,\nu}(r)$.
\begin{prop}
The set $\{T(\bold{m})\}$ for all $\bold{m}$ such that for all $(r,q,l) \ne (0,0,1)$  we have $m_{r,q,l} \in \mathbb Z_{\ge 0}$ and $m_{0,0,1}=0$, forms a basis of $\widetilde{\mathcal D}_{t,k,\nu}(r)$ (and of $\widetilde{\mathcal D}^{\rm{ext}}_{t,k,\nu}(r)$).
\end{prop}
\begin{proof}
Indeed, since $\widetilde{\Delta}$ is an isomorphism, it follows that the images of the basis of $U(\enalg[x,y])/(1_{\enalg}-\nu)$ form a basis of the DDCA.

Now let us pass to the associated graded with respect to the horizontal filtration. We know that ${\rm gr}_h(U(\enalg[x,y])/(1_{\enalg}- \nu)) = S^{\bullet}(\enalg[x,y]/\Bbbk \cdot 1_{\enalg})$. Hence the basis of this vector space is given by $\prod_{r,q,l} (\alpha_l x^ry^q)^{m_{r,q,l}}$ for all $\bold{m}$ specified in the statement of the problem. But now under ${\rm gr}_h(\widetilde{\Delta})$ these elements map exactly into the images of $T(\bold{m})$ in the associated graded ${\rm gr}_h({\rm gr}_v(\widetilde{\mathcal D}_{t,k,\nu}(r)))$. 

Hence $T({\bold{m}})$ form a basis of $\widetilde{\mathcal D}_{t,k,\nu}(r)$.
\end{proof}

\subsection{DDCA extended by a central element} \label{sectcentr}

In the previous section we have seen that $\widetilde{\mathcal D}_{t,k,\nu}(r)$ has a certain basis which arises from the fact that this DDCA is a deformation of $U(\enalg[x,y])/(1_{\enalg}-\nu)$. Here we would like to extend this construction to the case of $U(\enalg[x,y])$. We can do this if we turn $\nu$ into a central element instead of a scalar. 

In order to do this let us start with $\widetilde{\mathcal D}^{\rm{ext}}_{t,k,\nu}(r)$ which is an algebra over $\Cext$. If we can find a certain $\mathbb C[\nu]$-lattice in $\widetilde{\mathcal D}^{\rm{ext}}_{t,k,\nu}(r)$ which is closed under multiplication, this would allow us to consider this lattice as an algebra over $\mathbb C$, making $\nu$ a new central element.

In order to do this we need to show that the structure constants of $\widetilde{\mathcal D}^{\rm{ext}}_{t,k,\nu}(r)$ in the basis given by $T(\bold{m})$ are polynomial in $\nu$.

\begin{prop}
The product of $T(\bold{m}_1)$ and $T(\bold{m}_2)$ in $\widetilde{\mathcal D}^{\rm{ext}}_{t,k,\nu}(r)$ can be expressed as a linear combination of $T(\bold{m})$ with coefficients in $\mathbb C[\nu]$ for $\bold{m}$ such that $m_{0,0,1} = 0$.
\end{prop}
\begin{proof}
Since $T(\bold{m}_i) = \prodF T_{\nu_n}(\bold{m}_i)$ we can instead prove that for $n$ big enough $T_n(\bold{m}_1)\cdot T_n(\bold{m_2})$ can be expressed as a linear combination of $T_n(\bold{m}')$ with coefficients which depend polynomially on $n$. Recall that $T_n(\bold{m})$ are the elements of $ B_{t,k}(\nu,r)$.

In order to do so we will first introduce a notion of an admissible sum:
\begin{def0}
For a collection of functions $a:[\lambda]\to \{x,y\}$, $u: [\lambda] \to [k]$ and $\gamma: [k] \to [r^2]$, construct an element:
$$
\sum_{i_1,\dots,i_k = 1}^n (\alpha_{\gamma(1)})_{i_1}(\alpha_{\gamma(2)})_{i_2}\dots(\alpha_{\gamma(k)})_{i_k}a(1)_{i_{u(1)}}a(2)_{i_{u(2)}}\dots a(l)_{i_{u(l)}}\bold{e} \ .
$$
We will call all such elements admissible sums. Call $|k|$ the width and $|\lambda|$ the weight of the admissible sum.
\end{def0}

Note that the product of admissible sums is an admissible sum. Indeed if we have two admissible sums with the data $(a_1,u_1,\gamma_1)$ and $(a_2, u_2,\gamma_2)$, their product is an admissible sum with the function $a$ given by concatenation of $a_1$ and $a_2$, i.e. $a:[\lambda_1+\lambda_2] \to \{x,y\}$ such that $a(i) = a_1(i)$ for $i\le \lambda_1$ and $a(i) = a_2(i-\lambda_1)$ for $i > \lambda_1$; with the function $u$ given by concatenation of $u_1$ and $u_2$ in the sense that $u:[\lambda_1+\lambda_2]\to[k_1+k_2]$ maps $i \le \lambda_1$ to $u(i) = u_1(i)$ and maps $i> \lambda_1$ to $u(i) = u_2(i-\lambda_1)+k_1$; with the function $\gamma$ given by concatenation of $\gamma_1$ and $\gamma_2$, i.e. $\gamma: [k_1+k_2]\to [r^2]$, i.e. $\gamma(i) = \gamma_1(i)$ for $i\le k_1$ and $\gamma(i) = \gamma_2(i-k_1)$ for $i>k_1$. This follows easily from the fact that $(g)_i$ commutes with both $x_j$ and $y_j$. Note that we see that the weights and widths of the admissible sums add up when we take their product.

Now also note that we have:
$$
T_{r,q,n}(\alpha_l) = \frac{(r)!(q)!}{(r+q)!}\sum_{\substack{a:[r+q]\to \{x,y\}, \\ |a^{-1}(x)|=r}}\sum_{i=1}^n (\alpha_l)_i a(1)_i a(2)_i \dots a(r+q)_i \bold{e} \ .
$$
I.e., we see that $T_{r,q,n}(\alpha_l)$ is equal to the linear combination of admissible sums with width $1$ and weight $r+q$ with $n$-independent coefficients.

Since $T_n(\bold{m})$ is the linear combination of the products of $T_{r,q,n}(\alpha_l)$ with $n$-independent coefficients it follows that $T_n(\bold{m})$ itself is a linear combination of admissible sums with $n$-independent coefficients. Hence $T_n(\bold{m}_1)T_n(\bold{m}_2)$ is also such a linear combination.

Now if we prove that any admissible sum can written down as a linear combination of $T_n(\bold{m})$ with coefficients depending polynomially on $n$ for $\bold{m}$ such that $m_{0,0,1}= 0$, we would prove our Proposition.

Let us prove this result by inducting on the sum of the weight and the width of the admissible sum.

As the base of our induction suppose we have an admissible sum of weight $0$ and width $0$. Then the sum is just $1$, so we are done, since $T(\bold{m})$ with $m_{r,q,l} = 0$ for all $r,q,l$ is equal to $1$.

Now for the induction step suppose we have proven our hypothesis for all admissible sums with the sum of weight and width less than $N$.

Suppose we have an admissible sum $S$ of weight $\lambda$ and width $k$ given by functions $a,u,\gamma$, such that $\lambda+k=N$. First suppose that $\Img(u)$ does not cover the set $\gamma^{-1}(1)$. It follows that there is $j \in [k]$ such that $i_j$ does not appear as a subscript of $x$ or $y$ and only appears as a subscript of $(\alpha_{\gamma(j)})_{i_j}= (1)_{i_j}$, but $(1)_{i_j} = 1$, so we can take this sum, gaining a multiple of $n$ and reducing our problem to the admissible sum with smaller width, for which the problem is already solved. Hence in this case we are done.\footnote{Note that this is precisely where the polynomial dependence on $n$ comes from.}

So we can suppose that there are no $j \in [k]$ such that $\gamma(j)=1$ and $j \notin \Im(u)$. Now let us define $\bold{m}$ in the following way. Set $m_{r,q,l}=|\{j \in [k]| \ \gamma(j) = l, \ R_j = r,\  Q_j = q\}|$, where $R_j = |\{ i \in [\lambda]| \ u(i) = j , \ a(i)=x\}|$ and $Q_j = |\{ i \in [\lambda]| \ u(i) = j , \ a(i)=y\}|$. Notice that we have $m_{0,0,1} = 0$ by our requirement.  

Now note that $T_n(\bold{m})$ is proportional with an $n$-independent coefficient to the linear combination of admissible sums which differ from $S$ only by the permutation of $[\lambda]$ and $[k]$. If we prove that when we permute elements in the admissible sum the only extra terms we get are admissible sums with smaller sum of width and weight with $n$-independent coefficients, we are done. Indeed, then it would follow that for some $n$-independent constant $S-T_n(\bold{m})$ is the linear combination of admissible sums with the sum of weight and width $<N$ for which the hypothesis is known.

So let us prove this assertion. Since $(g)_i$ commutes with both $x_j$ and $y_j$, $x_i$ commute among themselves and $y_i$ commute too, we need to consider three cases: 1)what happens when we commute $(\alpha_l)_{i_j}$ and $(\alpha_{l'})_{i_{j'}}$ in the sum; 2)what happens when we commute $x_{i_j}$ and $y_{i_{j'}}$ in the sum; 3) what happens when we commute $x_{i_j}$ and $y_{i_j}$ in the sum.

In the first case we use the fact that $[(g)_{i_j},(h)_{i_{j'}}] = \delta_{i_j,i_{j'}}([g,h])_{i_j}$. So it follows that the extra term in the sum we get is as follows:
$$
\sum_{\dots, i_j,\dots, i_{j'},\dots=1}^n \dots \delta_{i_j,i_{j'}}([\alpha_l,\alpha_{l'}])_{i_j}\dots = \sum_{\dots, i_j, \dots, \cap{i_{j'}},\dots=1}^n \dots  ([\alpha_l,\alpha_{l'}])_{i_j} \dots \ .
$$
So in this case, since $[\alpha_l,\alpha_{l'}]$ can be written as a linear combination of $\alpha_i$ with $n$-independent coefficients, it follows that we get admissible sums with smaller width, as required.

In the second case we know that $[x_i,y_j] = \delta_{ij}(t-k\sum_{m\ne i} s_{im}\sigma_{im} -k) + (ks_{ij}\sigma_{ij})$, when we insert this into our sum somewhere, first of all the weight drops by two. Then in the first term, which is proportional to $\delta_{i_j,i_{j'}}(t-k\sum_{m\ne i_j}s_{i_j,m}\sigma_{i_j,m}-k)$, we delete the sum over $i_{j'}$ (this forces us to take the product of two $\alpha_l$ in the $\enalg$ part of the admissible sum after some commutation, but this by the above remarks doesn't cause a problem). Then we also are required to commute all $S_n$ elements to the right to be absorbed into $\bold{e}$, which only changes the function $u$ in the admissible sum, and to move all $\sigma$'s to the left, where by acting they permute $(\alpha_l)_i$, changing the function $\gamma$. The second term is proportional to $(ks_{i_{j},i_{j'}}\sigma_{i_j,i_{j'}})$. And here again we just commute $S_n$ elements to the right and $\sigma$'s to the left.

Now the final case is when we commute $x_{i_j}$ with $y_{i_j}$. Since, $[x_i,y_i] = t-k\sum_{m\ne i} s_{im}\sigma_{im}$, we again see that the weight drops by $2$ and all of the preceding remarks apply to make all extra terms into the linear combinations of admissible sums with lower width plus weight with $n$-independent coefficients.

Thus we have proven the induction step and the proposition follows.
\end{proof}

From this proposition it follows that the $\mathbb C[\nu]$-lattice given by $\bigoplus_{\bold{m}, m_{0,0,1}\ne 0} \mathbb C[\nu] \cdot T(\mu)$ forms a subalgebra in $\widetilde{\mathcal D}^{\rm{ext}}_{t,k,\nu}(r)$. So we can define:
\begin{def0}
Define the DDC algebra $\mathcal D_{t,k}(r)$ over $\mathbb C$ to be equal to the $\mathbb C[\nu]$-lattice $\bigoplus_{\bold{m}, m_{0,0,1}\ne 0} \mathbb C[\nu] \cdot T(\mu) \subset \widetilde{\mathcal D}^{\rm{ext}}_{t,k,\nu}(r)$.
\end{def0}

Now in this algebra $\nu$ becomes a central element which we will call $K$. Note that before we had $\prodF T_{0,0,\nu_n}(1) = \prodF \nu_n = \nu$. Now in this algebra it becomes $K$ -- an independent element, so it makes sense to also denote $T_{0,0}(1) = K \in \mathcal D_{t,k}(r)$.

We can also see that this extends the isomorphism
$$
\widetilde{\Delta}:U(\enalg[x,y])/(1_{\enalg}-\nu) \simeq {\rm gr}_{v}(\widetilde{\mathcal D_{t,k,\nu}(r)})
$$
to  the isomorphism:
$$
\widetilde{\Delta}: U(\enalg[x,y]) \simeq {\rm gr}_v(\mathcal D_{t,k}(r)) \ ,
$$

which fully explains the name "deformed double current algebra".

Thus we can conclude that:
\begin{cor}
The set $\{T(\bold{m})\}$ for all $\bold{m}$  forms a basis of ${\mathcal D}_{t,k}(r)$.
\end{cor}

Notice that we also have the following important Corollary which connects the DDC algebra $\mathcal D_{t,k}(r)$ with $\widetilde{\mathcal D}_{t,k,\nu}(r)$.

\begin{cor}
The DDC algebra $\widetilde{\mathcal D}_{t,k,\nu}(r)$ is isomorphic to $\mathcal D_{t,k}(r)/(K - \nu)$.
\end{cor}

\section{Guay's construction.}

The DDC algebras were constructed by Guay and co-authors first for type A in \cite{Gu2} and then for any simple Lie algebra in \cite{guay2016deformed}. In this section we will explain how our algebra is connected with the one constructed by Guay. Note that in this section we always have $r\ge 4$, since Guay's DDC algebras are not defined (yet) for smaller rank.

\subsection{Guay's DDCA of type A.}

Here we will recall one of the main definitions of Guay's DDCA. 

\begin{def0} \label{Guaydef}
The Guay's DDC algebra $ \mathbb D_{\lambda,\beta}(r)$ for $\lambda,\beta \in \mathbb C$ is an algebra generated by elements $z, K(z), Q(z), P(z)$, where $z \in \mathfrak{sl}_r$, which satisfy the following relations. The subalgebra generated by $z$ and $K(z)$ is isomorphic to $U(\mathfrak{sl}_r[u])$, i.e. there is a map $U(\mathfrak{sl}_r[u]) \to  \mathbb D_{\lambda,\beta}(r)$. Similarly the subalgebra generated by $z$ and $Q(z)$ is isomorphic to $U(\mathfrak{sl}_r[v])$. Also, $P(z)$ is linear in $z$ and $[y,P(z)] = P([y,z])$. And if we consider $1 \le a,b,c,d\le r$ such that $(a,b) \ne (d,c)$ and $a\ne b, c \ne d$ we have:
$$
[K(E_{ab}),Q(E_{cd})] = P([E_{ab},E_{cd}]) + (\beta - \tfrac{\lambda}{2})(\delta_{bc}E_{ad} + \delta_{ad}E_{bc}) + \frac{\lambda}{4}(\delta_{ad} + \delta_{cb}) S(E_{ab},E_{cd}) +
$$
$$
+\frac{\lambda}{4}\sum_{1 \le i\ne j < n}S([E_{ab},E_{ij}],[E_{ji}, E_{cd}]) \ ,
$$
where $S(z,y) = zy+yz$.
\end{def0}
\begin{rem}
Note that if $[E_{ab},E_{cd}] = 0$ (i.e. $b \ne c$ and $a \ne d$) the last relation simplifies to:
$$
[K(E_{ab}),Q(E_{cd})] = - \lambda E_{ad}E_{cb} \ ,
$$
since only the last term for $(i,j) = (b,d)$ or $(i,j) = (c,a)$ survives. 
\end{rem}

\subsection{Construction of the homomorphism $ \mathbb D_{\lambda,\beta}(r) \to {\mathcal D}_{t,k}(r)$}

In this section we will construct a map from  Guay's DDCA to our DDCA, i.e. we will construct the elements in ${\mathcal D}_{t,k}(r)$ that satisfy the relations of Definition \ref{Guaydef}.

First we need to establish a convenient way to perform calculations in ${\mathcal D}_{t,k}(r)$. Since this algebra is defined as an ultraproduct of a family of other algebras, we will use that for our calculations.
\begin{def0}
Suppose $X \in \widetilde{\mathcal D}_{t,k,\nu}(r)$ is an element of DDCA. We have $X = \prodF X_n$, where $X_n \in B_{t_n,k_n}(\nu_n,r)$. We will denote this correspondence by $X \sim X_n$. 

A similar correspondence exists for $\mathcal D_{t,k}(r)$. The only difference is that here $\nu_n \sim K$ instead of $\nu_n \sim \nu$.

Note also that all the elements of $B_{t_n,k_n}(\nu_n,r)$ actually have a multiple of $\bold{e}$ on the right. We will omit this for brevity.
\end{def0}

Now we can construct a map between the DDC algebras.
\begin{prop} \label{psicons}
There is a map $\psi:  \mathbb D_{k,-\frac{t}{2} - \frac{k(r-2)}{4}}(r) \to \mathcal {D}_{t,k}(r)$ given by:
$$
\psi(z) = T_{0,0}(z) \ , \ \psi(K(z)) = T_{1,0}(z) \ , \ \psi(Q(z)) = T_{0,1}(z) \ , \ \psi(P(z)) = T_{1,1}(z) \ ,
$$
where $z \in \mathfrak{sl}_r$.
\end{prop}
\begin{proof}
The above expressions can be rewritten as 
$$
\psi(z) \sim \sum_{i=1}^{\nu_n}(z)_i \ , \ \psi(K(z)) \sim \sum_{i=1}^{\nu_n}(z)_i\cdot x_i \ , \ \psi(Q(z)) \sim \sum_{i=1}^{\nu_n}(z)_i\cdot y_i \ \  \textrm{and}  
$$
$$
\psi(P(z)) \sim \sum_{i=1}^{\nu_n} (z)_i \cdot \frac{x_iy_i+y_ix_i}{2} \ .
$$

We only need to check that the images of $z,K(z),Q(z)$ and $P(z)$ satisfy the required relations.

We will start with relations between $z$ and $K(z)$. Note that their image in ${\mathcal D}_{t,k}(r)$ is contained in the ultraproduct of subalgebras of $B_{t_n,k_n}(\nu_n,r)$ generated by $x_i$ and $g$. Since $x_i$ commute with each other, these subalgebras are equal to 
$$
(\enalg^{\otimes \nu_n} \otimes \Bbbk [x_1,\dots,x_{\nu_n}])^{S_{\nu_n}} = ((\enalg[x])^{\otimes \nu_n})^{S_{\nu_n}} = S^{\nu_n}(\enalg[x]) \ .
$$
But by Proposition \ref{Deltaprop}  we know that 
$$
\prodF S^{\nu_n}(\enalg[x]) \simeq U(\enalg[x])/(1_{\enalg[x]} - \nu) \ .
$$
And the construction of this isomorphism also shows that under it $\psi(z) \mapsto z$ and \linebreak $\psi(K(z)) \mapsto z\cdot x$. Also note that:
$$
[\psi(K(z_1)),\psi(K(z_2))] \sim \sum_{i,j}[(z_1)_i  \cdot x_i, (z_2)_j \cdot x_j ] \mapsto \sum_{i}([z_1,z_2])_i (x_i)^2 \sim [z_1,z_2]\cdot x^2 \ 
$$
under the above isomorphism.
Thus these elements generate

$$
U(\mathfrak{sl}_r[x]) \subset U(\enalg[x])/(1_{\enalg[x]} - \nu) .
$$

The same holds for $z$ and $Q(z)$.

Now, since $\psi(P(z)) \sim \sum_{i=1}^{\nu_n} (z)_i \cdot \frac{x_iy_i + y_ix_i}{2}$ it follows that it is linear in $z$ and 
$$
[\psi(y),\psi(P(z))] = \psi(P[y,z]) \ ,
$$
since $[(y)_j, (z)_i] = \delta_{ij}([y,z])_{i}$.

We need to check the last relation of Definition \ref{Guaydef}. 

So let us first calculate $[\psi(K(E_{ab})), \psi(Q(E_{cd}))]$. We have:
$$
[\psi(K(E_{ab})), \psi(Q(E_{cd}))] \sim 
$$
$$
\sim [\sum_{i=1}^{\nu_n} (E_{ab})_i \cdot x_i, \sum_{j=1}^{\nu_n} (E_{cd})_j \cdot y_j] = \sum_{i,j =1}^{\nu_n}[(E_{ab})_i,(E_{cd})_j]\cdot x_iy_j + \sum_{i,j = 1}^{\nu_n} (E_{cd})_j (E_{ab})_i \cdot [x_i,y_j] = 
$$
$$
= \sum_{i =1}^{\nu_n}([E_{ab},E_{cd}])_i \cdot \left(\frac{x_iy_i+y_ix_i+[x_i,y_i]}{2}\right)+
$$
$$
+\sum_{i,j=1}^{\nu_n} \left(\frac{(E_{cd})_j(E_{ab})_i + (E_{ab})_i (E_{cd})_j-[(E_{ab})_i, (E_{cd})_j]}{2}\right)\cdot [x_i,y_j] = 
$$
$$
= \psi(P([E_{ab},E_{cd}])) + \sum_{i,j=1}^{\nu_n}\frac{(E_{cd})_j(E_{ab})_i + (E_{ab})_i (E_{cd})_j}{2} \cdot [x_i,y_j] \ .
$$
Now we need to work with the last term. We will expand it using the commutator relation in the extended Cherednik algebra and we note that elements of $S_{\nu_n}$ disappear into the assumed $\bold{e}$ term in the formula:
$$
\sum_{i,j=1}^{\nu_n}\frac{(E_{cd})_j(E_{ab})_i + (E_{ab})_i (E_{cd})_j}{2} \cdot [x_i,y_j] =
$$
$$
=-\sum_{i\ne j}\frac{(E_{cd})_j(E_{ab})_i - (E_{ab})_i(E_{cd})_j}{2} k_n\sigma_{ij} + \sum_i \frac{(E_{cd})_i(E_{ab})_i + (E_{ab})_i(E_{cd})_i}{2} (t_n - k_n\sum_{m \ne i}\sigma_{im}) \ .
$$
Notice that $(E_{\alpha\beta})_i (E_{\gamma\delta})_j \sigma_{ij} = (E_{\alpha \delta})_i (E_{\gamma\beta})_j$, so the above expression becomes:
$$
-k_n \sum_{i\ne j}(E_{ad})_i (E_{cb})_j - \frac{t_n}{2}\sum_i (\delta_{ad}(E_{cb})_i + \delta_{bc}(E_{ad})_i) +\frac{k_n}{2}\sum_{i\ne m}(\delta_{ad}(E_{cb})_i\sigma_{im} + \delta_{bc}(E_{ad})_i\sigma_{im}) \ .
$$

Now we need to calculate what $\sum_{i\ne m}(E_{\alpha\beta})_i \sigma_{im}$ is equal to. We have:
$$
\sum_{i\ne m}(E_{\alpha\beta})_i \sigma_{im} = \sum_{i\ne m}\sum_{\gamma,\delta}(E_{\alpha\beta})i (E_{\gamma\delta})_i(E_{\delta\gamma})_m = \sum_{i\ne m}\sum_{\delta}(E_{\alpha\delta})_i(E_{\delta\beta})_m \ .
$$
So the answer is:
$$
[\psi(K(E_{ab})), \psi(Q(E_{cd}))] \sim 
$$
$$
\sim \psi(P([E_{ab},E_{cd}])) - k_n \sum_{i,j=1, i\ne j}^{\nu_n}(E_{ad})_i (E_{cb})_j - \frac{t_n}{2}\sum_i^{\nu_n} (\delta_{ad}(E_{cb})_i + \delta_{bc}(E_{ad})_i) -
$$
$$
+\frac{k_n}{2}\sum_{m, i = 1, m \ne i}^{\nu_n}\sum_{e=1}^{r}\left(\delta_{ad}(E_{ce})_i(E_{eb})_m + \delta_{bc}(E_{ae})_i(E_{ed})_m\right) \ .
$$

Now we need to calculate the image of the r.h.s. of the same relation. The first term is clear. The second term contains elements like $\psi(E_{\alpha\beta}) \sim \sum_i (E_{\alpha\beta})_i$. The third term is more complex. We have:
$$
\psi(S(E_{ab},E_{cd})) \sim \sum_{i,j}((E_{ab})_i(E_{cd})_j + (E_{cd})_j(E_{ab})_i) = 2\sum_{i\ne j}(E_{ab})_i(E_{cd})_j + \sum_i (\delta_{bc}(E_{ad})_i + \delta_{ad}(E_{cb})_i) \ .
$$
Now we want to transform the last term:
$$
\sum_{\alpha \ne \beta} S([E_{ab},E_{\alpha\beta}],[E_{\beta\alpha},E_{cd}]) \ . 
$$
Before we calculate its image we can rewrite it as follows:
$$
-2S(E_{ad},E_{cb}) + \delta_{ad}\sum_{\alpha\ne a}S(E_{\alpha b},E_{c\alpha}) + \delta_{bc}\sum_{\alpha \ne b}S(E_{a\alpha},E_{\alpha d}) \ .
$$
Now, since $c \ne d$ and $a \ne b$, it follows that:
$$
\psi(S(E_{ad},E_{cb}))\sim 2\sum_{i\ne j}(E_{ad})_i(E_{cb})_j \ .
$$

And since $\delta_{ad}\delta_{bc} = 0$ in our situation, it follows:
$$
\psi(\delta_{ad}\sum_{\alpha \ne a}S(E_{\alpha b},E_{c\alpha})) \sim 2\delta_{ad}\sum_{i \ne j}\sum_{\alpha \ne a} (E_{\alpha b})_i(E_{c\alpha})_j + (r-1)\delta_{ad}\sum_i(E_{cb})_i \ ,
$$
and similarly:
$$
\psi(\delta_{bc}\sum_{\alpha \ne b}S(E_{a\alpha},E_{\alpha d})) \sim 2\delta_{bc}\sum_{i \ne j}\sum_{\alpha \ne b} (E_{a\alpha})_i(E_{\alpha d})_j + (r-1)\delta_{bc}\sum_i(E_{ad})_i \ .
$$

Now we can assemble all the formulas to obtain that the r.h.s. of the relation equals to:
$$
\psi(P([E_{ab},E_{cd}])) +\left[\beta-\frac{\lambda}{2}\right]\left(\delta_{ad}\sum_i (E_{ad})_i + \delta_{ad}\sum_i (E_{cb})_i\right) + \frac{\lambda}{2}(\delta_{ad}+\delta_{bc})\sum_{i\ne j}(E_{ab})_i(E_{cd})_j +
$$
$$
+\frac{\lambda}{4}\left(\sum_{i} \delta_{bc}(E_{ad})_i + \sum_i\delta_{ad}(E_{cb})_i\right) - \lambda\sum_{i\ne j}(E_{ad})_i(E_{cb})_j +
$$
$$
+\frac{\lambda}{2}\left(\delta_{ad}\sum_{i\ne j}\sum_{\alpha \ne a}(E_{\alpha b})_i(E_{c\alpha})_j + \delta_{bc}\sum_{i\ne j}\sum_{\alpha \ne b}(E_{a\alpha})_i(E_{\alpha d})_j\right)+
$$
$$
+ \frac{\lambda (r-1)}{4}\left(\delta_{ad}\sum_{i}(E_{cb})_i + \delta_{bc}\sum_i (E_{ad})_i\right) = 
$$
$$
=\psi(P([E_{ab},E_{cd}])) + \left[\beta - \frac{\lambda}{2}+ \frac{\lambda}{4}+ \frac{\lambda(r-1)}{4}\right]\left(\delta_{ad}\sum_i (E_{ad})_i + \delta_{ad}\sum_i (E_{cb})_i\right) -  
$$
$$
- \lambda \sum_{i\ne j}(E_{ad})_i(E_{cb})_j +\frac{\lambda}{2}\left(\delta_{ad}\sum_{i\ne j}\sum_{\alpha }(E_{\alpha b})_i(E_{c\alpha})_j + \delta_{bc}\sum_{i\ne j}\sum_{\alpha }(E_{a\alpha})_i(E_{\alpha d})_j\right) \ .
$$
We can see that these two formulas are the same if $\lambda =  k$ and $\beta = -\frac{t}{2} -\frac{k}{4}(r-2)$.
\end{proof}

From now on fix $\lambda = k$ and $\beta = -\frac{t}{2} - \frac{k}{4}(r-2)$.

\subsection{Surjectivity of $\psi$}

In this section we would like to prove that $\psi$ is in fact a surjective map. 

\begin{prop} \label{psisurj}
For $t+rk \ne 0$, the map $\psi$ defined in Proposition \ref{psicons} is surjective.
\end{prop}

\begin{proof}
Since $T(\mathbf{m})$ form a basis of $\mathcal D_{t,k}(r)$ and they themselves are given by the linear combinations of the products of $T_{r,q}(z)$ for all $z \in \enalg$ it follows that it is enough to prove that $T_{r,q}(z)$ lie in the image of $\psi$. More precisely to prove that $\psi$ surjects onto $F_v^{N}\mathcal D_{t,k}(r)$ it is enough to prove that all $T_{r,q}(z)$ for $r+q \le N$ are in the image of $\psi$. 

We would like to prove the last statement by inducting on $N$. But our induction will be slightly more involved than one could hope for.

Nevertheless we would like to start with proving the base case. Namely that all $T_{0,0}(z)$ are in the image of $\psi$. Indeed we know that for all $z \in \mathfrak{sl}_r$ $\psi(z) = T_{0,0}(z)$, so we only need to show that $K = T_{0,0}(\rm{Id})$ is in the image. Denote $H=E_{11}-E_{22}$ and consider $[\psi(K(H)),\psi(Q(H))]$:
$$
[\psi(K(H)),\psi(Q(H))] \sim  \sum_{i,j}(H)_i (H)_j [x_i,y_j] \ .
$$

Now we will calculate this modulo the image of $\psi$ (we will denote this by $\underset{\psi}{\sim}$). So after we apply the same operations to the last term as in Proposition \ref{psicons} and then note that $H^2 = E_{11}+E_{22}$, we have:
$$
[\psi(K(H)),\psi(Q(H))] \sim -k_n \sum_{i\ne j}(H)_i(H)_j\sigma_{ij} - t_n\sum_i (E_{11}+E_{22})_i + k_n\sum_{i\ne j}(E_{11}+E_{22})_i \sigma_{ij} \ .
$$
This time we will calculate the last terms by inserting the identity $1 = (E_{11} + \dots E_{rr})_k$. We get:
$$
\sum_{i\ne j}(E_{11}+E_{22})_i \sigma_{ij} = \sum_{i\ne j}(E_{11})_i(E_{11})_j + \sum_{i\ne j}(E_{22})_i(E_{22})_j + \sum_{\alpha \ne 1}\sum_{i\ne j}(E_{1\alpha})_i(E_{\alpha 1})_j + \sum_{\alpha \ne 2}\sum_{i\ne j}(E_{2\alpha})_i(E_{\alpha 2})_j \ .
$$
Putting this into original formula we get:
$$
[\psi(K(H)),\psi(Q(H))] =
$$
$$
 = -k_n\sum_{i\ne j}\left((E_{11})_i (E_{11})_j + (E_{22})_i(E_{22})_j -(E_{12})_i(E_{21})_j - (E_{21})_i(E_{12})_j\right) - t_n\sum_i (E_{11}+E_{22})_i -
 $$
 $$
 +k_n\sum_{i\ne j}\left((E_{11})_i(E_{11})_j +(E_{22})_i(E_{22})_j\right) + k_n\sum_{\alpha \ne 1}\sum_{i\ne j}(E_{1\alpha})_i(E_{\alpha 1})_j+k_n\sum_{\alpha \ne 2}\sum_{i\ne j}(E_{2\alpha})_i(E_{\alpha 2})_j \ .
$$

Now notice that $\psi(z_1)\psi(z_2) = \sum_{i\ne j}(z_1)_i(z_2)_j + \sum_i (z_1\cdot z_2)_i$, so it follows that:
$$
[\psi(K(H)),\psi(Q(H))] \underset{\psi}{\sim} -(t_n +rk_n)\sum_{i}(E_{11}+E_{22})_i= -(t_n +rk_n)\sum_{i}(2 + z)_i  \underset{\psi}{\sim} -2(t+rk)K \ ,
$$
for some $z \in \mathfrak{sl}_r$. Hence we know that $K$ is in the image of $\psi$.

Now we will prove the surjectivity in general by induction. For each $m$ we will be proving that all $T_{r,q}(z)$ with $r+q \le m+1$ and $z \in \mathfrak{sl}_r$ are in the image of $\psi$ and that all $T_{r,q}(1)$ with $r+q \le m$ are in the image of $\psi$ (so that $\psi$ surjects onto $F^{m}_v\mathcal D_{t,k}(r)$). From this statement it will follow that $\psi$ is surjective. 

Now the base for $m=0$ holds since we have just proved that $T_{0,0}(1) = K$ is in the image of $\psi$ and also we know that $T_{0,1}(z)$ and $T_{1,0}(z)$ are in the image. 

So we are ready to prove the induction step. Suppose the statement holds for $m$ and we want to prove it for $m+1$. We need to prove that all $T_{r,q}(z)$ with $z \in \mathfrak{sl}_r$ and $r+q=m+2$ lie in the image of $\psi$ and also that all $T_{r,q}(1)$ for $r+q=m+1$ lie there.

We will start with the first statement. Note that we already know that $T_{m+2,0}(z)$ and $T_{0,m+2}(z)$ are in the image, since $\psi(z)$ and $\psi(K(z))$ generate $U(\mathfrak{sl}_r[x]) \subset \mathcal D_{t,k}(r)$ and the analogous statement holds for $\psi(z)$ and $\psi(Q(z))$. It is enough to prove that, for example $T_{m+2-k,k}(E_{13})$ is in the image for each $k$ from $1$ to $m+1$, since then by taking commutators with $T_{0,0}(z)$ we can obtain any other $T_{m+2-k,k}(z')$. Let's calculate the commutator of $T_{m+2-k,k-1}(E_{12})$ and $T_{0,1}(H)$, both of which are in the image. To do that, we will denote by $f_{r,q}(i)$ the polynomial in $x_i$ and $y_i$ which appears in $T_{r,q}(E_{12}) \sim \sum_i (E_{12})_i f_{r,q}(i)$. We have:
\scriptsize
$$
[T_{m+2-k,k-1}(E_{12}), T_{0,1}(H)] \sim
$$
$$
\sim -2\sum_i (E_{12})_i \frac{f_{m+2-k,k-1}(i)y_i + y_if_{m+2-k,k-1}(i)}{2} + \sum_{i,j}\frac{(E_{12})_i(H)_j + (H)_j(E_{12})_i}{2}[f_{m+2-k,k-1}(i),y_j] \ .
$$
\normalsize
Now since we know that $\psi$ surjects onto $F^m_v\mathcal D_{t,k}(r)$ we would like to calculate the above commutator modulo degree $m$. The last term is zero modulo degree $m$ since it contains at least one commutator of $x$ and $y$ which decreases the degree by $2$. Now also modulo degree $m$ the monomials in the first term commute. So, we have:
$$
[T_{m+2-k,k-1}(E_{12}), T_{0,1}(E_{23})] \underset{\psi}{\sim} -2\sum_i (E_{12})_i x^{m+2-k}_iy_i^k \underset{\psi}{\sim} -2T_{m+2-k,k}(E_{13}) \ .
$$

Now we only need to prove that $T_{r,q}(1)$ for $r+q = m+1$ are in the image. To do that let us calculate the commutator of $T_{r,q+1}(H)$ and $T_{1,0}(H)$. We have:
$$
[T_{1,0}(H),T_{r,q+1}(H)] \underset{\psi}{\sim} \sum_{i,j}(H)_i(H)_j [x_i, f_{r,q+1}(j)] \ .
$$
We need to calculate this term modulo degree $m$. Hence we can commute the terms in $f_{r,q+1}(j)$ under the commutator. I.e. we have:
$$
[T_{1,0}(H),T_{r,q+1}(H)] \underset{\psi}{\sim} \sum_{i,j}(H)_i(H)_j [x_i, x_j^ry_j^{q+1}] \underset{\psi}{\sim} \sum_{i,j}(H)_i(H)_j x_j^r[x_i, y_j^{q+1}] \underset{\psi}{\sim} \ .
$$
$$ 
\underset{\psi}{\sim} \sum_{i,j}\sum_{l=0}^{q}(H)_i(H)_j x_j^ry_j^l[x_i,y_j]y_j^{q-l} \underset{\psi}{\sim} 
$$
$$
 \underset{\psi}{\sim} -k_n\sum_{i\ne j}\sum_{l=0}^q (H)_i(H)_j \sigma_{ij} x_j^ry_j^l y_i^{q-l} - t_n(q+1)\sum_{i}(H^2)_i x_i^ry_i^q  +k_n\sum_{i\ne j}\sum_{l=0}^q(H^2)_j \sigma_{ij} x_j^{r}y_j^ly_i^{q-l}  \underset{\psi}{\sim}
$$
$$
 \underset{\psi}{\sim} -k_n\sum_{l=0}^q\sum_{i\ne j}\left[ (E_{11})_i (E_{11})_j + (E_{22})_i(E_{22})_j - (E_{12})_j(E_{21})_i - (E_{21})_j(E_{12})_i\right]x_j^ry_j^ly_i^{q-l} -
 $$
 $$
 - t_n(q+1)\sum_i(H^2)_ix_i^ry_i^q +
 $$
 $$
 +k_n\sum_{l=0}^q\sum_{i\ne j}\left((E_{11})_i (E_{11})_j + (E_{22})_i(E_{22})_j + \sum_{\alpha \ne 1}(E_{1\alpha})_j(E_{\alpha1})_i + \sum_{\alpha\ne 2}(E_{2\alpha})_j(E_{\alpha2})_i\right)x_j^ry_j^ly_i^{q-l}\underset{\psi}{\sim}
$$
$$
\underset{\psi}{\sim} -t_n(q+1)\sum_i(H^2)_ix_i^ry_i^q +
$$
$$
+k_n\sum_{l=0}^q\sum_{i\ne j}\left((E_{12})_j (E_{21})_i + (E_{21})_j(E_{12})_j + \sum_{\alpha \ne 1}(E_{1\alpha})_j(E_{\alpha1})_i + \sum_{\alpha\ne 2}(E_{2\alpha})_j(E_{\alpha2})_i\right)x_j^ry_j^ly_i^{q-l} \ .
$$
Now note the following formula:
$$
\sum_{i\ne j}(z_1)_j(z_1)_ix_j^{r_1}y_j^{q_1}x_i^{r_2}y_i^{q_2} \sim 
$$
$$
\sim T_{r_1,q_1}(z_1)T_{r_2,q_2}(z_2) - \sum_i (z_1 \cdot z_2) x_i^{r_1+r_2}y_i^{q_1+q_2} \ \textrm{modulo} \ F^{r_1+q_1+r_2+q_2-1}_v\mathcal D_{t,k}(r) \ .
$$
In our case $r_1+q_1+r_2+q_2 = r+q =m +1$. Since we know that $F^{m}_v\mathcal D_{t,k}(r) \subset \Img (\psi)$, it follows that we can use this formula. Also notice that since everywhere there we can use the above formula $z_1,z_2 \in \mathfrak{sl}_r$ and $r_i+q_i < m+1$, it follows that $T_{r_i,q_i}(z_i) \in \Img(\psi)$. Thus it follows that:
$$
[T_{1,0}(H),T_{r,q+1}(H)] \underset{\psi}{\sim}
$$
$$
\underset{\psi}{\sim} -t_n(q+1)\sum_i(H^2)_ix_i^ry_i^q 
-k_n(q+1)\sum_{i}\left[(E_{11})_i  + (E_{22})_i + (r-1)(E_{11})_i+ (r-1)(E_{22})_i\right]x_i^ry_i^q \underset{\psi}{\sim}  \ .
$$
$$
\underset{\psi}{\sim} -(q+1)(t_n + rk_n)\sum_i (E_{11}+E_{22})_i x_i^ry_i^q \underset{\psi}{\sim} -2(q+1)(t_n + rk_n)\sum_i x_i^ry_i^q \underset{\psi}{\sim} -2(q+1)(t_n + rk_n)T_{r,q}(1) \ .
$$

So we have proven the inductive step and hence it follows that $\psi$ is surjective.
\end{proof}

\subsection{Injectivity of $\psi$}

In this subsection we are going to show that, if $t+rk \ne 0$, $\psi$ is injective and, hence, it is an isomorphism. In order to do that we will show that $\mathbb D_{\lambda,\beta}(r)$ has a faithful representation $\mathbb D_{\lambda,\beta}(r)\to \End(M)$, such  that $M$ is also a $\mathcal D_{t,k}(r)$-module and the action map for $\mathbb D_{\lambda,\beta}(r)$ factors through $\psi$.

Here, we will extensively use the results of \cite{Gu2}. First of all we need to define an alternative presentation of Guay's DDCA  -- $  D_{\lambda,\beta}(r)$, which will be isomorphic to $\mathbb D_{\lambda,\beta}(r)$. This presentation is quite involved and its exact form isn't important for us, so we will state an abbreviated version of it.
\begin{def0} [Definition 8.1 in \cite{Gu2}]
The algebra $  D_{\lambda,\beta}(r)$ is generated by elements $X^\pm_{i,0}, X^\pm_{i,1}, H_{i,0}, H_{i,1}$ for $i \in \{1,\dots, r-1\}$ and $X_{0,0}^+, X_{0,1}^{+,\pm}$, which satisfy a number of relations.

Also there are two specific elements in this algebra, denoted by $\omega_{0}^{+,\pm}$ (see Section 9 of \cite{Gu2}).
\end{def0}

Another result which is important to us is the explicit structure of the isomorphism between $ D_{\lambda,\beta}(r)$ and $\mathbb{D}_{\lambda,\beta}(r)$.
\begin{thm}[Theorem 15.1 in \cite{Gu2}]
Define a map $\zeta:{D}_{\lambda,\beta}(r) \to \mathbb D_{\lambda,\beta}(r)$ to be equal to:
$$
\zeta(X^{\pm}_{i,0})= E^{\pm}_i , \ \zeta(H_{i,0})= H_i , \ \zeta(X^{\pm}_{i,1})= Q(E^{\pm}_i) , \ \zeta(H_{i,1})= Q(H_i) ,
$$
$$
\zeta(X_{0,0}^{+}) = K(E_{- \theta}), \ X_{0,1}^{+,\pm} = P(E_{-\theta}) -\lambda\omega_0^{+,\pm} \ ,
$$
where $E^+_i = E_{i,i+1}$, $E^-_{i} = E_{i+1,i}$, $E_{\theta} = E_{1,r}$, $E_{-\theta} = E_{r,1}$ and $H_i = E_{i,i} - E_{i+1,i+1}$. This map is an isomorphism.
\end{thm}

Another set of results that Guay proved in \cite{Gu2} are concerned with constructing a family of  $D_{\lambda,\beta}(r)$-modules. 
\begin{prop} [Section 9 of \cite{Gu2}]
For any $l \in \mathbb Z_{\ge 0}$ the vector space \linebreak $\mathbf{V}_l = H_{-t,-k}(l,1) \otimes_{\mathbb C[S_l]} (\mathbb C^r)^{\otimes l}$ has a structure of $  D_{\lambda,\beta}(r)$-module given by the following formulas. For $m \in \mathbf{v} \in H_{-t,-k}(l,1) \otimes_{\mathbb C[S_l]} (\mathbb C^r)^l$ we have
$$
X^{\pm}_{i,r}(m \otimes \mathbf{v}) = \sum_{j=1}^{l} my_j^r \otimes (E^{\pm}_i)_{j}\mathbf{v} , \ H_{i,r}(m \otimes \mathbf{v}) = \sum_{j=1}^{l} my_j^r \otimes (H_i)_{j}\mathbf{v} , 
$$
$$
X_{0,0}^{+}(m \otimes \mathbf{v}) = \sum_{j=1}^{l} mx_j \otimes (E_{-\theta})_{j}\mathbf{v} ,
$$
$$
 X_{0,1}^{+, \pm}(m \otimes \mathbf{v}) = \sum_{j=1}^{l} m\frac{x_jy_j + y_jx_j}{2} \otimes (E_{-\theta})f_{j}\mathbf{v} - \lambda \omega_0^{+,\pm}(m \otimes \mathbf{v}) \ .
$$
\end{prop}
\begin{rem}
Note that in our case the parameters of the Cherednik algebra has to be $-t,-k$ as opposed to Guay's $t,k$. This discrepancy arises from us using a different sign in one of the commutators which define the Cherednik algebra, and also because of the different signs in the formulas which connect $t,k$ and $\lambda,\beta$ in our paper. We will also see that these signs arise naturally because in the definition above we are using a right action on the Cherednik algebra side of the tensor product.
\end{rem}
Guay also proved a PBW property for his DDCA in \cite{Gu2}. As a by-product of his proof he arrived at the following result.
\begin{prop} \label{PBW}
For $\beta \ne \frac{r\lambda}{4} + \frac{\lambda}{2}$ (equivalently $t + rk \ne 0$) and for any $x \in   D_{\lambda,\beta}(r)$, there exists $l \in \mathbb Z_{> 0}$  such that the map $\rho_l:  D_{\lambda,\beta}(r) \to \End(\mathbf{V}_l)$ specified above sends $x$ to a non-zero operator, i.e. $\rho_l(x) \ne 0$.
\end{prop}

In other words it follows that $\bigoplus_{l > 0} \mathbf{V}_l$ gives us a faithful representation of $ D_{\lambda,\beta}(r)$. 

Now to prove that $\psi$ is injective we will construct a $\mathcal D_{t,k}(r)$-module structure on $\mathbf{V}_l$. In order to do that we first want to show that for any $l$ there is a surjective map from $\mathcal D_{t,k}(r)$ to $B_{t,k}(l,r)$.
\begin{prop}
There is a surjective map $\pi_l: \mathcal D_{t,k}(r) \to B_{t,k}(l,r)$ that sends
$$
T(\mathbf m) \mapsto T_l(\mathbf m)
$$
including $K \mapsto l$.
\end{prop}
\begin{proof}
Since $T(\mathbf m)$ form a basis, these formulas define a vector space map from $\mathcal D_{t,k}(r)$ to $ B_{t,k}(l,r)$.
Now from Section \ref{sectcentr} we know that a product of $T_l(\mathbf m_1)$ and $T_l(\mathbf m_2)$ is a linear combination of $T_l(\mathbf m)$ with coefficients being polynomial in $l$. And the same statement holds for $T(\mathbf m)$ but we need to substitute $K$ for $l$ in these polynomials. Hence this map is a map of algebras. It is surjective since $T_l(\mathbf m)$ form a generating set of $B_{t,k}(l,r)$.

\end{proof}

Now let us first construct a representation of $H_{t,k}(l,r)$ on $H_{-t,-k}(l,1) \otimes (\mathbb C^r)^l$. To do this we use the same ideas as in Proposition \ref{extcherstrep}.
\begin{prop}
For any $l \in \mathbb Z_{>0}$, there is a structure of representation of $H_{t,k}(l,r)$ on $H_{-t,-k}(l,1) \otimes (\mathbb C^r)^l$ given by:
$$
x_i(m \otimes \mathbf{v}) = mx_i \otimes \mathbf{v} , \ y_i(m \otimes \mathbf{v}) = my_i \otimes \mathbf{v} , \ (g)_i(m \otimes \mathbf{v}) = m \otimes (g)_i\mathbf{v} ,
$$
$$
s_{ij}(m\otimes v) = ms_{ij}\otimes \sigma_{ij}\mathbf{v} \ .
$$
\end{prop}
\begin{proof}
We just need to check that these  formulas define a representation. This is easy to do. Indeed for example:
$$
[y_i,x_i](m\otimes \mathbf{v}) = m[x_i,y_i]\otimes \mathbf{v} = m(t -k\sum_{j\ne i}s_{ij})\otimes \mathbf {v} = 
$$
$$
t m\otimes \mathbf{v} - k\sum_{j \ne i}m s_{ij} \otimes \sigma_{ij}^2\mathbf{v} = (t - k\sum_{j \ne i}s_{ij}\sigma_{ij})(m \otimes \mathbf{v}) \ .
$$
There we can see that the opposite signs for $t$ and $k$ come from the use of the right action. The other commutators can be checked in the similar fashion.
\end{proof}
Note that we can derive the following Corollary from this result:
\begin{cor}
For any $l \in \mathbb Z_{>0}$, there is a structure of a representation of $B_{t,k}(l,r)$ on $\mathbf {V}_l = H_{-t,-k}(l,1) \otimes_{\mathbb C[S_l]} (\mathbb C^r)^l$ obtained by restriction of the representation of $H_{t,k}(l,r)$ on $H_{-t,-k}(l,1) \otimes (\mathbb C^r)^l$.
\end{cor}
 We will denote the corresponding map by $\tau_l:B_{t,k}(l,r) \to  \End(\mathbf{V}_l)$.
\begin{proof}
Indeed this follows from the fact that $B_{t,k}(l,r) = \bold{e}H_{t,k}(l,r)\bold{e}$ and the fact that the action of $\mathbb C[S_l] \subset H_{t,k}(l,r)$ on $H_{-t,-k}(l,1) \otimes (\mathbb C^r)^l$ is right on $H_{-t,-k}(l,1)$ and left on $(\mathbb C^r)^l$. Hence the averaging operator $\bold{e}$ ensures that we stay within $\mathbf{V}_l$.
\end{proof}

It follows that we have the following diagram:

\vspace{0.1cm}
\begin{center}
\includegraphics[width=\textwidth/2]{comm.jpg}
\end{center}
\vspace{0.1cm}

We want to show that this diagram is commutative:
\begin{prop} \label{commut}
For any $l \in \mathbb Z_{> 0}$ it holds that $\rho_l = \tau_l \circ \pi_l \circ \psi \circ \zeta$.
\end{prop}
\begin{proof}
It is enough to check this identity on the generators of $ D_{\lambda,\beta}(r)$. This is easy to do. We have
$$
(\pi_l \circ \psi \circ \zeta)(X_{0,r}^{\pm}) = (\pi_l \circ \psi)(E_i^{\pm}) = \pi_l (T_{0,0}(E_i^{pm}) = T_{0,0,l}(E_i^{pm})  ,
$$
and hence:
$$
(\tau_l \circ \pi_l \circ \psi \circ \zeta)(X_{0,r}^{\pm})(m \otimes \mathbf{v}) = \sum_{j} m \otimes (E_i^{\pm})_j\mathbf {v} = \rho_l(X_{0,r}^{\pm})(m \otimes \mathbf{v}) \ .
$$
And the same holds for $H_{i,0}$. 

Now $(\pi_l \circ \psi \circ \zeta)(X_{1,r}^{\pm}) = T_{0,1,l}(E_i^{\pm})$, hence
$$
(\tau_l \circ \pi_l \circ \psi \circ \zeta)(X_{1,r}^{\pm})(m \otimes \mathbf{v}) = \sum_{j} my_j \otimes (E_i^{\pm})_j\mathbf {v} = \rho_l(X_{1,r}^{\pm})(m \otimes \mathbf{v}) \ .
$$
And again the same holds for $H_{i,1}$.

For $X_{0,0}^+$ we have $(\pi_l \circ \psi \circ \zeta)(X_{0,0}^{+}) = T_{1,0,l}(E_{-\theta})$ and so:
$$
(\tau_l \circ \pi_l \circ \psi \circ \zeta)(X_{0,0}^{+})(m \otimes \mathbf{v}) = \sum_{j} mx_j \otimes (E_{-\theta})_j\mathbf {v} = \rho_l(X_{0,0}^{+})(m \otimes \mathbf{v}) \ .
$$
Lastly 
$$
(\pi_l \circ \psi \circ \zeta)(X_{0,1}^{+,\pm}) = T_{1,1,l}(E_{-\theta}) - \lambda(\pi_l \circ \psi \circ \zeta)\omega_0^{+, \pm} \ .
$$
Now since $\omega_0^{+,\pm}$ lies in the subspace generated by $X_{i,0}^{\pm}$ and $H_{i,0}$ it follows that 
$$(\tau_l \circ \pi_l \circ \psi \circ \zeta)(\omega_0^{+,\pm}) = \rho_l(\omega_0^{+,\pm})
$$
holds as proved by the previous formulas.
Hence we have:
$$
(\tau_l \circ \pi_l \circ \psi \circ \zeta)(X_{0,1}^{+,\pm })(m \otimes \mathbf{v}) = \sum_{j} m\frac{x_jy_j + y_jx_j}{2} \otimes (E_{-\theta})_j\mathbf {v} - \lambda \omega_0^{+,\pm}(m \otimes \mathbf{v}) = \rho_l(X_{0,1}^{+,\pm})(m \otimes \mathbf{v}) \ .
$$
And so the result follows.
\end{proof}

And so we can formulate the result which we wanted to prove in this section.
\begin{thm}
For $t+kr \ne 0$, the map $\psi: \mathbb D_{\lambda,\beta}(r) \to \mathcal D_{t,k}(r)$ constructed in Proposition \ref{psicons} is an isomorphism.
\end{thm}
\begin{proof}
We know surjectivity from Proposition \ref{psisurj}. Now take a non-zero $x \in \mathbb D_{\lambda,\beta}(r)$. Since $\zeta$ is an isomorphism there is $y \in  D_{\lambda,\beta}(r)$ such that $\zeta(y) = x$. Now by Proposition \ref{PBW} there exists $l$ such that $\rho_l(y) \ne 0$. Hence by Proposition \ref{commut} it follows that $(\tau_l \circ \pi_l \circ \psi \circ \zeta)(y) = (\tau_l \circ \pi_l)(\psi(x)) \ne 0$. Hence $\psi(x) \ne 0$. Thus $\psi$ is injective and so it is an isomorphism.
\end{proof}

\begin{rem}
We also expect that it is possible to construct a direct isomorphism with another presentation of DDC algebra introduced by Kevin Costello in \cite{costello2017holography}. The existence of this isomorphism is posed as a question in section 2.1 of the same paper.

There, Costello considers quantum Hamiltonian reduction of a certain Nakajima quiver variety. More explicitly, he considers the vector space 
$$
V_{N,K} = \mathfrak{gl}_K\oplus \Hom(\mathbb C^N, \mathbb C^K) \ ,
$$
with an action of both $\mathrm{GL}(K)$ and $\mathrm{GL}(N)$, and defines $\mathcal M_{N,K}$ to be the symplectic reduction with respect to the $\mathrm{GL}(K)$:
$$
\mathcal M^c_{N,K} = T^*V_{N,K}//\mathrm{GL(K)} \ ,
$$
where we have subtracted a $c$-multiple of the identity from the moment map.
So, the quantum Hamiltonian reduction is given by:
$$
\mathcal O_h(\mathcal M^c_{N,K}) =(\textrm{D}(V_{N,K})/I_c)^{\textrm{GL}(K)} \ ,
$$
where $I_c$ is the left ideal in the algebra of differential operators. This ideal is generated by the images of elements of $\mathfrak{gl}_K$ under the moment map (i.e. $\mu(x) -cTr(x)$ for $x \in \mathfrak{gl}_K$) and it becomes two-sided once we take the invariants.

Finally, to define the DDCA itself, Costello considers a certain limit of these algebras with $K$ going to infinity, to get 
$$
\mathcal O_h(\mathcal M^c_{N,\bullet}) \ .
$$

What we expect is that this algebra is isomorphic to $\mathcal D_{t,k}(r)$ with $r=N$ and $t$ and $k$ being certain functions of $c$ and $h$. This could be proven by making rigorous the following sketch of an argument. First, one can construct $\mathcal O_h(\mathcal M^c_{N,\bullet})$ directly in the Deligne category $\Rep(\textrm{GL}(\nu))$ by considering $\textrm{D}\left(\mathfrak{gl}_\nu \oplus \Hom(\mathbb C^N, \mathbb C^\nu)\right)$ (both $\mathfrak{gl}_\nu$ and $\mathbb C^\nu$ can be defined as objects of Deligne category) and then by taking the quotient by the generalization of $I_c$ and the invariants ($\Hom_{\Rep(\textrm{GL}(\nu))}(\Bbbk,\cdot)$, where $\Bbbk$ stands for the unit object of this category). On the other hand it can also be constructed as an ultraproduct in $n$, $\prodF \mathcal O_h(\mathcal M^c_{N,n})$. Costello does this implicitly in his paper by considering a certain class of admissible sequences of elements of these algebras. Now to finish the proof of the isomorphism with $\mathcal D_{t,k}(r)$ it would be enough to know that 
$$
\left(\textrm{D}\left(\mathfrak{gl}_\nu \oplus \Hom(\mathbb C^N, \mathbb C^\nu)\right)/I_c\right)^{\textrm{GL}(\nu)}
$$
is isomorphic to the spherical algebra of the extended Cherednik algebra. But this should follow through an ultraproduct argument from the following conjecture by P. Etingof that extends the deformed Harish-Chandra isomorphism to the case of the extended Cherednik algebra (see \cite{etingof2002symplectic} for the case of $r=1$).

\begin{conj} [P.Etingof]
For any $n,r \in \mathbb N$ and $t,k \in \mathbb C$ there is an isomorphism between the spherical sublagebra $B_{t,k}(n,r)$ of the extended Cherednik algebra and the quantization of the Hamiltonian reduction of $\mathcal M^c_{n,r}$, $\left(\textrm{D}\left(\mathfrak{gl}_n \oplus \Hom(\mathbb C^r, \mathbb C^n)\right)/I_c\right)^{\textrm{GL}(n)}$ for some values of $c,h\in \mathbb C$.
\end{conj}

\end{rem}

\end{document}